\documentclass[a4paper,12pt]{amsart}
\usepackage[margin=2cm]{geometry}
\usepackage{amscd,amsmath,amsthm}
\usepackage{fouriernc} 
\usepackage[mathscr]{euscript}
\usepackage{stmaryrd}
\usepackage{MnSymbol} 
\usepackage{enumerate}
\usepackage{color}
\usepackage[pdfusetitle,unicode]{hyperref}
\usepackage[cmtip,all]{xy} 
\usepackage{enumitem}
\usepackage{datetime}
\usepackage{setspace}					
\usepackage{microtype}					
\usepackage{graphicx}					
\usepackage{epsfig}						
\usepackage{verbatim}
\usepackage[english]{babel}		
\usepackage[utf8]{inputenc} 				
\usepackage{enumitem}					
\usepackage{xspace} 						
\hypersetup{linktoc=page}
\def\resp{{\sfcode`\.1000 resp.}}
\def\ie{{\sfcode`\.1000 i.e.}}
\def\eg{{\sfcode`\.1000 e.g.}}

\def\cf{{\sfcode`\.1000 cf.}}
\def\Cf{{\sfcode`\.1000 Cf.}}

\newcommand{\mathletter}[1]{\mathbf{#1}}
\newcommand{\N}{\mathletter{N}}				
\newcommand{\Z}{\mathletter{Z}}				
\newcommand{\Q}{\mathletter{Q}}				
\newcommand{\F}{\mathletter{F}}				
\newcommand{\Pe}{\mathletter{P}}
\renewcommand{\P}{\Pe}

\newcommand{\A}{\mathletter{A}}
\newcommand{\Bb}{\mathletter{B}}

\renewcommand{\to}[1][]{\overset{#1}{\rightarrow}}		
\newcommand{\To}[1][]{\overset{#1}{\longrightarrow}}	
\newcommand{\inj}[1][]{\overset{#1}{\hookrightarrow}}		
\newcommand{\surj}[1][]{\overset{#1}{\twoheadrightarrow}}		
\newcommand{\oT}[1][]{\overset{#1}{\longleftarrow}} 			
\newcommand{\iso}[1][]{\underset{#1}{\To[\cong]}}		 	

\newcommand{\inv}{^{-1}}								
\newcommand{\Colon}{\,\colon\,}
\newcommand{\skp}[1]{\langle #1\rangle}				

\DeclareMathOperator{\Hom}{Hom}				
\DeclareMathOperator{\Fun}{Fun}				
\DeclareMathOperator{\coker}{coker}			
\DeclareMathOperator{\Ho}{H}					
\DeclareMathOperator{\GL}{GL}				
\DeclareMathOperator{\Bl}{Bl}				

\DeclareMathOperator{\K}{K}						
\DeclareMathOperator{\BGL}{BGL} 					

\DeclareMathOperator{\KH}{KH}					

\newcommand{\Kcont}{\K^\mathrm{cont}}			
\newcommand{\on}{\,\,\mathrm{on}\,\,}			

\newcommand{\df}[1]{\textbf{#1}}					
\newcommand{\etale}{\'{e}tale}    					
\newcommand{\infcat}{$\infty$-category}
\newcommand{\infcats}{$\infty$-categories}
\newcommand{\inftop}{$\infty$-topos}

\newcommand{\Cech}{\v{C}ech}

\renewcommand{\phi}{\varphi}
\renewcommand{\epsilon}{\varepsilon}
\renewcommand{\rho}{\varrho}

\theoremstyle{definition}				
	\newtheorem{defin}{Definition}[section]
	\newtheorem{example}[defin]{Example}
	\newtheorem{remark}[defin]{Remark}
	\newtheorem{caveat}[defin]{Caveat}

	\newtheorem{reminder}[defin]{Reminder}
	\newtheorem*{notation*}{Notation}
	
\theoremstyle{plain} 					
	\newtheorem{thm}[defin]{Theorem}
	\newtheorem*{thm*}{Theorem}
	\newtheorem{prop}[defin]{Proposition}
	\newtheorem{lemma}[defin]{Lemma}
	\newtheorem{cor}[defin]{Corollary}
	\newtheorem*{cor*}{Corollary}
	
	\newtheorem{conj}[defin]{Conjecture}
		\newcounter{zaehler}
	
	\newtheorem{introthm}[zaehler]{Theorem}
	\newtheorem{introcor}[zaehler]{Corollary}
\theoremstyle{remark} 					

\newcounter{formulanumber}[defin]

\newcommand{\vquad}{\vspace{6pt}}
\newcommand{\absatz}{\vquad\noindent}

\DeclareMathOperator*{\colim}{colim}

\DeclareMathOperator{\fib}{fib}

\DeclareMathOperator*{\prolim}{``lim''}

\renewcommand{\O}{\mathcal{O}}		
			
\newcommand{\Acal}{\mathcal{A}}

\newcommand{\Ccal}{\mathcal{C}}
\newcommand{\Dcal}{\mathcal{D}}

\newcommand{\Ical}{\mathcal{I}}

\newcommand{\Ucal}{\mathcal{U}}
\newcommand{\Vcal}{\mathcal{V}}
\newcommand{\Xcal}{\mathcal{X}}

\newcommand{\na}{nonarchimedean}

\newcommand{\Anod}{{A_0}}

\newcommand{\cdh}{\mathrm{cdh}}
\newcommand{\Hcdh}{\Ho_{\cdh}}

\newcommand{\Zar}{\mathrm{Zar}}
\newcommand{\Hzar}{\Ho_{\Zar}}

\newcommand{\sing}{\mathrm{sing}}
\newcommand{\Hsing}{\Ho_{\sing}}

\newcommand{\biZar}{\mathrm{biZar}}
\newcommand{\Hbizar}{\Ho_{\biZar}}
\newcommand{\rh}{\mathrm{rh}}
\newcommand{\Hrh}{\Ho_{\rh}}

\DeclareMathOperator{\Ll}{L}
\renewcommand{\L}{\Ll}

\newcommand{\Lrh}{\L_\rh}

\DeclareMathOperator{\aL}{a}
\newcommand{\arh}{\aL_\rh}

\newcommand{\op}{^\mathrm{op}}
\newcommand{\aff}{^\mathrm{aff}}

\newcommand{\ad}{^\mathrm{ad}}

\newcommand{\berk}{^\mathrm{berk}}

\newcommand{\red}{_\mathrm{red}}

\newcommand{\mathcat}[1]{\mathrm{#1}}
\newcommand{\Ab}{\mathcat{Ab}}

\newcommand{\Sch}{\mathcat{Sch}}
\newcommand{\Schqc}{\Sch^\mathrm{qc}}

\newcommand{\Vect}{\mathcat{Vec}}

\newcommand{\Rig}{\mathcat{Rig}}

\newcommand{\FSch}{\mathcat{FSch}}
\newcommand{\Mdf}{\mathcat{Mdf}}
\newcommand{\BL}{\mathcat{Bl}}
\newcommand{\Adm}{\mathcat{Adm}}

\newcommand{\Shh}{\mathcat{Sh}}

\newcommand{\mathinfcat}[1]{\mathbf{#1}}
\newcommand{\Spc}{\mathinfcat{Spc}}
\newcommand{\Sp}{\mathinfcat{Sp}}
\newcommand{\Sh}{\mathinfcat{Sh}}
\newcommand{\PSh}{\mathinfcat{PSh}}
\newcommand{\Pro}{\mathinfcat{Pro}}

\newcommand{\Perf}{\mathinfcat{Perf}}

\DeclareMathOperator{\Spec}{Spec}		
\DeclareMathOperator{\Spm}{Spm}			
\DeclareMathOperator{\Spb}{Spb}			
\DeclareMathOperator{\Spf}{Spf}			
\DeclareMathOperator{\Spa}{Spa}			

\newcommand{\carre}[4]{\[\begin{xy}\xymatrix{#1\ar[r]\ar[d]&#2\ar[d]\\#3\ar[r]&#4}\end{xy}\]}
\newcommand{\carrelift}[4]{\[\begin{xy}\xymatrix{#1\ar[r]\ar[d]&#2\ar[d]\\#3\ar[r]\ar@{.>}[ur]&#4}\end{xy}\]}
\newcommand{\carretag}[5][]{\[\tag{#1}\begin{xy}\xymatrix{#2\ar[r]\ar[d]&#3\ar[d]\\#4\ar[r]&#5}\end{xy}\]}

\newcommand{\lift}[3]{\[\begin{xy}\xymatrix{&#1\ar[d]\\#2\ar[r]\ar@{.>}[ur]&#3}\end{xy}\]}
\newcommand{\liftmap}[6]{\[\begin{xy}\xymatrix{&#1\ar[d]^{#2}\\#3\ar@{.>}[ur]^{#4}\ar[r]_{#5}&#6}\end{xy}\]}

\reversemarginpar		
\newcommand{\stacks}[1]{\cite[\href{http://stacks.math.columbia.edu/tag/#1}{Tag #1}]{stacks-project}}

\title{Continuous K-Theory and Cohomology of Rigid Spaces}
\author{Christian Dahlhausen}
\urladdr{cdahlhausen.eu}
\thanks{
\!\!\!\!\!\!\!\!\emph{Address:} Mathematisches Institut, Im Neuenheimer Feld 205, 69120 Heidelberg, Germany.\\
\emph{Email:} cdahlhausen@mathi.uni-heidelberg.de\\
\emph{ORCID-ID:} 0000-0002-0478-969X
}
\hypersetup{
pdftitle={Continuous K-Theory and Cohomology of Rigid Spaces},
pdfauthor={Christian Dahlhausen},
pdfkeywords={K-Theory, Algebraic Geometry, Homotopy Theory},
pdfmenubar=true
}

\begin{document}
\begin{abstract}
We establish a connection between continuous K-theory and integral cohomology of rigid spaces.
Given a rigid analytic space over a complete discretely valued field, its continuous K-groups vanish in degrees below the negative of the dimension. Likewise, the cohomology groups vanish in degrees above the dimension.
The main result provides the existence of an isomorphism between the lowest possibly non-vanishing continuous K-group and the highest possibly non-vanishing cohomology group with integral coefficients.
A key role in the proof is played by a comparison between cohomology groups of an admissible Zariski-Riemann space with respect to different topologies; namely, the rh-topology which is related to K-theory as well as the Zariski topology whereon the cohomology groups in question rely.
\end{abstract}
\maketitle
\tableofcontents

	\section{Introduction}
The negative algebraic K-theory of a scheme is related to its singularities. If $X$ is a regular scheme, then $\K_{-i}(X)$ vanishes for $i>0$. For an arbitrary noetherian scheme $X$ of dimension $d$ we know that
\begin{enumerate}
	\item $\K_{-i}(X)=0$ for $i>d$,
	\item $\K_{-i}(X) \cong\K_{-i}(X\times\A^n)$ for $i\geq d$, $n\geq 1$, and
	\item $\K_{-d}(X) \cong \Hcdh^d(X;\Z)$.
\end{enumerate} 
The affine case of (i) was a question of Weibel \cite[2.9]{weibel-analytic} who proved (i) and (ii) for $d\leq 2$ \cite[2.3, 2.5, 4.4]{weibel-normal}. For varieties in characteristic zero (i)-(iii) were proven by Corti\~nas-Haesemeyer-Schlichting-Weibel \cite{chsw08} and the general case is due to Kerz-Strunk-Tamme \cite{kst}. As an example for the lowest possibly non-vanishing group $\K_{-d}(X)$, the cusp $C=\{y^2=x^3\}$ over a field has $\K_{-1}(C)=0$ whereas the node $N = \{y^2=x^3+x^2\}$ over a field (of characteristic not 2) has $\K_{-1}(N) = \Z$; more generally, for a nice curve the rank is the number of loops \cite[2.3]{weibel-normal}. 
The main result of this article are analogous statements of (i)-(iii) for continuous K-theory of rigid analytic spaces in the sense of Morrow \cite{morrow}.

\vquad\noindent There is a long history of versions of K-theory for topological rings that take the topology into account. For instance, the higher algebraic K-groups of a ring $A$ can be defined via the classifying space $\BGL(A)$ of the general linear group $\GL(A)$. If $A$ happens to be a Banach algebra over the complex numbers, it also makes sense to consider $\GL(A)$ as a topological group and to define \emph{topological K-theory} $\K^\mathrm{top}(A)$ analogously in terms of the classifying space $\BGL^\mathrm{top}(A)$. This yields a better behaved K-theory for complex Banach algebras which satisfies homotopy invariance and excision (which does not hold true in general for algebraic K-theory). Unfortunately, a similar approach for \emph{\na} algebras does not behave well since the \na{} topology is totally disconnected. Karoubi-Villamayor \cite{kv71} and Calvo \cite{calvo} generalised topological K-theory to arbitrary Banach algebras (either \na{} or complex) in terms of the ring of power series converging on a unit disc.
A different approach is to study \df{continuous K-theory} which is the pro-spectrum 
	\begin{align*} 
	\Kcont(R) = \prolim_n\, \K(R/I^n)
	\end{align*}
where $R$ is an $I$-adic ring with respect to some ideal $I\subset R$ (\eg{} $\Z_p$ with the $p$-adic topology or $\F_p\llbracket t\rrbracket$ with the $t$-adic topology). Such ``continuous'' objects have been studied amply in the literature -- \cf{} Wagoner \cite{wagoner-cont, wagoner-cont-dvr},  Dundas \cite{dundas-cont}, Geisser-Hesselholt \cite{geisser-hesselholt-K-TC, geisser-hesselholt-K-F_p}, or Beilinson \cite{beilinson} -- and they were related by Bloch-Esnault-Kerz to the Hodge conjecture for abelian varieties \cite{bek-zero} and the p-adic variational Hodge conjecture \cite{bek-p-adic}. Morrow \cite{morrow} suggested an extension of continuous K-theory to rings $A$ admitting an open subring $A_0$ which is $I$-adic with respect to some ideal $I$ of $A_0$ (\eg{} $\Q_p = \Z_p[p\inv]$ or $\F_p(\!(t)\!) = \F_p\llbracket t\rrbracket[t\inv]$).\footnote{Actually, Morrow only talks about affinoid algebras.} This notion was recently studied by Kerz-Saito-Tamme \cite{kst18} and they showed that it coincides in non-positive degrees with the groups studied by Karoubi-Villamayor and Calvo. For an affinoid algebra $A$ over a discretely valued field, Kerz proved the corresponding analytical statements to (i) and (ii); that is replacing algebraic K-theory by continuous K-theory and the polynomial ring by the ring of power series converging on a unit disc \cite{kerz-icm}. Morrow showed that continuous K-theory extends to a sheaf of pro-spectra on rigid $k$-spaces for any discretely valued field $k$. The main result of this article provides analogous statements of (i)-(iii) above for continuous K-theory of rigid $k$-spaces; the statements (i) and (ii) extend Kerz' result to the global case and statement (iii) is entirely new.

\begin{introthm}[Theorem~\ref{global-main-thm}, Theorem~\ref{Kcont-vanishing--thm}]
Let $X$ be a quasi-compact and quasi-separated rigid $k$-space of dimension $d$ over a discretely valued field $k$. Then:
\begin{enumerate}
	\item For $i> d$ we have $\Kcont_{-i}(X)=0$.
	\item For $i\geq d$ and $n\geq 0$ the canonical map
	\[ \Kcont_{-i}(X) \to \Kcont_{-i}(X\times\mathbf{B}_k^n) \]
	is an isomorphism where $\mathbf{B}_k^n:=\Spm(k\skp{t_1,\ldots,t_n})$ is the rigid unit disc.
	\item If $d\geq 2$ or if there exists a formal model of $X$ which is algebraic (Definition~\ref{algebraic-formal--def}, \eg{} $X$ is affinoid or projective), then there exists an isomorphism 
	\[
	\Kcont_{-d}(X) \cong \Ho^d(X;\Z)
	\]
where the right-hand side is sheaf cohomology with respect to the admissible topology on the category of rigid $k$-spaces.
\end{enumerate}
\end{introthm}

There are several approaches to \na{} analytic geometry. Our proof uses rigid analytic spaces in the sense of Tate \cite{tate71} and adic spaces introduced by Huber \cite{huber94}. Another approach is the one of Berkovich spaces \cite{berkovich90} for which there is also a version of our main result as conjectured in the affinoid case by Kerz \cite[Conj.~14]{kerz-icm}.

\begin{introcor}[Corollary~\ref{global-main-thm-cor}]
Let $X$ be a quasi-compact and quasi-separated rigid analytic space of dimension $d$ over a discretely valued field. Assume that $d\geq 2$ or that there exists a formal model of $X$ which is algebraic (\eg{} $X$ is affinoid or projective). Then there is an isomorphism 
	\[
	\Kcont_{-d}(X) \cong \Ho^d(X\berk;\Z)
	\]
where $X\berk$ is the Berkovich space associated with $X$. 
\end{introcor}

If $X$ is smooth over $k$ or the completion of a $k$-scheme of finite type, then there is an isomorphism 
	\[
	\Ho^d(X\berk;\Z) \cong \Hsing^d(X\berk;\Z)
	\]
with singular cohomology by results of Berkovich \cite{berkovich99} and Hrushovski-Loeser \cite{hrushovski-loeser}. The identification of Corollary B is very helpful since it is hard to actually compute K-groups whereas the cohomology of Berkovich spaces is amenable for computations. For instance, the group $\Ho^d(X\berk;\Z)$ is finitely generated since $X\berk$ has the homotopy type of a finite CW-complex; such a finiteness statement is usually unknown for K-theory.

\absatz
An important tool within the proof of Theorem A is the \df{admissible Zariski-Riemann space} $\skp{X}_U$ which we will associate, more generally, with every quasi-compact and quasi-separated scheme $X$ with open subscheme $U$. The admissible Zariski-Riemann space $\skp{X}_U$ is given by the limit of all $U$-modifications of $X$ in the category of locally ringed spaces (Definition~\ref{U-admissible_ZR-space--defin}). In our case of interest where $A$ is a reduced affinoid algebra and $A^\circ$ is its open subring of power-bounded elements, then we will set $X=\Spec(A^\circ)$ and $U=\Spec(A)$. 
We shall relate its Zariski cohomology to the cohomology with respect to the so-called \df{rh-topology}, \ie{} the minimal topology generated by the Zariski topology and abstract blow-up squares (Definition~\ref{rh-cdh-topology--defin}). To every topology $\tau$ on the category of schemes (\eg{} Zar, Nis, rh, cdh), there is a corresponding appropriate site $\Sch_\tau(\skp{X}_U)$ for the admissible Zariski-Riemann space (Definition~\ref{ZR-site--defin}). We show the following statement which is later used in the proof of Theorem A and which is the main new contribution of this article.

\begin{introthm}[Theorem~\ref{Hzar=Hrh--thm}]
For every constant abelian rh-sheaf $F$ on $\Sch(\skp{X}_U)$ the canonical map
	\[
	\Hzar^*(\skp{X}_U\setminus U;F) \To \Hrh^*(\skp{X}_U\setminus U;F)
	\]
is an isomorphism. In particular,
	\begin{align*}
		\colim_{X' \in \Mdf(X,U)} \Hzar^*(X'\setminus U;F) = \colim_{X' \in \Mdf(X,U)} \Hrh^*(X'\setminus U;F).
	\end{align*}
where $\Mdf(X,U)$ is the category of all $U$-modifications of $X$ and $X'\setminus U$ is equipped with the reduced scheme structure. The same statement also holds if one replaces `Zar' by `Nis' and `rh' by `cdh'.
\end{introthm}

We also show an rh-version of a cdh-result of Kerz-Strunk-Tamme \cite[6.3]{kst}. This is not a new proof but the observation that the analogous proof goes through. The statement will enter in the proof of Theorem A.

\begin{introthm}[Theorem~\ref{Lrh_K=KH--thm}]
Let $X$ be a finite dimensional noetherian scheme. Then the canonical maps of rh-sheaves with values in spectra on $\Sch_X$
	\[
	\Lrh\K_{\geq 0} \To \Lrh\K \To \KH
	\]
are equivalences.
\end{introthm}

\vquad\noindent\textbf{Proofsketch for the main result.} \label{proofsketch}
 We shall briefly sketch the proof of Theorem A(iii) in the affinoid case (Theorem~\ref{main-thm}). For every reduced affinoid algebra $A$ and every model $X'\to\Spec(A^\circ)$ over the subring $A^\circ$ of power-bounded elements with pseudo-uniformiser $\pi$ there exists a fibre sequence \cite[5.8]{kst18}
	\[ \tag{$\star$}
	\K(X'\on\pi) \To \Kcont(X') \To \Kcont(A).
	\]
\noindent Now let us for a moment assume that $A$ is regular and that resolution of singularities is available so that we could choose a regular model $X'$ whose special fibre $X'/\pi$ is simple normal crossing. In this case, $\K(X'\on\pi)$ vanishes in negative degrees and hence we have 
\[
\Kcont_{-d}(A) \underset{(1)}{\cong} \Kcont_{-d}(X') \underset{(2)}{\cong} \K_{-d}(X'/\pi) \underset{(3)}{\cong} \Hcdh(X'/\pi;\Z)
\] 	
where (1) follows from $\Kcont_{-i}(X'\on\pi)=0$ for $i>0$, (2) from nil-invariance of K-theory in degrees $\geq d$ (Lemma~\ref{K_-d_is_nil-invariant--lemma}), and (3) is (iii) above \resp{} \cite[Cor. D]{kst}. 
Let $(D_i)_{i\in I}$ be the irreducible components of $X'/\pi$. As $X'/\pi$ is simple normal crossing, all intersections of the irreducible components are regular, hence their cdh-cohomology equals their Zariski cohomology which is $\Z$ concentrated in degree zero; hence $\Hcdh^d(X'/\pi;\Z)$ can be computed by the \Cech{} nerve of the cdh-cover $D:=\bigsqcup_{i\in I} D_i \to X/\pi$. On the other hand, the Berkovich space $\Spb(A)$ associated with $A$ is homotopy equivalent to its skeleton which is homeomorphic to the intersection complex $\Delta(D)$ \cite[2.4.6, 2.4.9]{nicaise}. Putting these together yiels $\Kcont_{-d}(A) \cong \Ho^d(\Spb(A);\Z) \cong \Ho^d(\Spa(A,A^\circ);\Z)$.
	
\vquad\noindent In the general case where $X'$ is an aribtrary model, we have do proceed differently. For $n<0$ and $\alpha\in\K_n(X'\on\pi)$ there exists by Raynaud-Gruson's \emph{platification par \'eclatement} an admissible blow-up $X''\to X'$ such that the pullback of $\alpha$ vanishes in $\K_n(X''\on\pi)$ \cite[7]{kerz-icm}. In the colimit over all models this yields that $\Kcont_n(A) \cong \Kcont_n(\skp{A_0}_A)$. For $d=\dim(A)$ we have $\Kcont_d(\skp{A_0}_A) \cong \K_d(\skp{A_0}_A/\pi)$ and the latter is isomorphic to $\Hrh^d(\skp{A_0}_A/\pi;\Z)$ via a descent spectral sequence argument (Theorem~\ref{K_-d=Hrh^d--thm}). Using Theorem C (Theorem~\ref{Hzar=Hrh--thm}) we can pass to Zariski cohomology. Now the result follows from identifying $\skp{A^\circ}_A$ with the adic spectrum $\Spa(A,A^\circ)$ (Theorem~\ref{adic_space_is_ZR--thm}).

\vquad\noindent\textbf{Leitfaden.}
In section~\ref{sec-continuous-k-theory} we recall the definition of and some basic facts about continuous K-theory. Then we introduce admissible Zariski-Riemann spaces in section~\ref{sec-ZR-spaces} and we establish a comparison between their rh-cohomology and their Zariski cohomology (Theorem~\ref{Hzar=Hrh--thm}) in section~\ref{sec-cohomology-ZR}. Subsequently we recall the connection between formal Zariski-Riemann spaces and adic spaces in section \ref{sec-formal-zr-adic}; this causes the adic spaces showing up in the main result. In section~\ref{sec-proof-main-result} we prove the main result in the affine case (Theorem~\ref{main-thm}). In section~\ref{sec-global-continuous-k-theory} we present, following Morrow, a global version of continuous K-theory and in section~\ref{sec-global-main-result} we prove the main result in the global case (Theorem~\ref{global-main-thm}). Finally, there is an Appendix~\ref{sec-rh-topology} about the rh-topology and rh-versions of results for the cdh-topology.

\vquad\noindent\textbf{Notation.}
Discrete categories are denoted by $\mathrm{upright}$ $\mathrm{letters}$ whereas genuine \infcats{} are denoted by $\mathbf{bold}$ $\mathbf{letters}$. We denote by $\Spc$ the \infcat{} of spaces \cite[1.2.16.1]{htt} and by $\Sp$ the \infcat{} of spectra \cite[1.4.3.1]{ha}.
Given a scheme $X$ we denote by $\Sch_X$ the category of separated schemes of finite type over $X$. If $X$ is noetherian, then every scheme in $\Sch_X$ is noetherian as well.

\vquad\noindent\textbf{Acknowledgements.} 
This article is a condensed version of my PhD thesis. I am thankful to my advisors Moritz Kerz and Georg Tamme for their advise and their constant support. I thank Johann Haas for many helpful discussions on the subject and Florent Martin for explaining to me Berkovich spaces and their skeleta. Concerning the expositon, I thank my thesis' referees Moritz Kerz and Matthew Morrow for their feedback as well as Georg Tamme for helpful comments on a draft of the present paper. 
The author was supported by the RTG/GRK 1692 ``Curvature, Cycles, and Cohomology'' and the CRC/SFB 1085 ``Higher Invariants'' (both Universität Regensburg) funded by the DFG and the Swiss National Science Foundation.
Further thanks go to the Hausdorff Research Institute for Mathematics in Bonn for generous hospitality during the Hausdorff Trimester Program on ``K-theory and related fields'' during the months of May and June in 2017.

\vquad\noindent\textbf{Data availability statement.} 
Data sharing not applicable to this article as no datasets were generated or analysed during the current study.

\vquad\noindent\textbf{Conflict of interest statement.} 
The author states that there is no conflict of interest.

	\section{Continuous K-theory for Tate rings}
	\label{sec-continuous-k-theory}

In this section we recall the definition of continuous K-theory as defined by Morrow~\cite{morrow} and further studied by Kerz-Saito-Tamme~\cite{kst18}.

\begin{defin} \label{k-theory--def}
Let $X$ be a scheme. We denote by $\K(X)$ the nonconnective K-theory spectrum $\K(\Perf(X))$ \`a la Blumberg-Gepner-Tabuada associated with the \infcat{} $\Perf(X)$ of perfect complexes on $X$ \cite[\S 7.1, \S 9.1]{bgt13}. For a ring $A$, we write $\K(A)$ denoting $\K(\Spec(A))$. For $i\in\Z$ we denote by $\K_i(X)$ and $\K_i(A)$ the $i$-th homotopy group of $\K(X)$ and $\K(A)$, respectively.
\end{defin}

\begin{remark}
For a scheme $X$ the homotopy category $\mathrm{Ho}(\Perf(X))$ is equivalent to the derived category of perfect complexes $\mathrm{Perf}(X)$ and the K-theory spectrum $\K(X)$ is equivalent to the one constructed by Thomason-Trobaugh \cite[\S 3]{tt90}. Every relevant scheme in this article is quasi-projective over an affine scheme, hence admits an ample family of line bundles. Thus K-theory can be computed in terms of the category $\Vect(X)$ of vector bundles (\ie{} locally free $\O_X$-modules). In view of Bass' Fundamental Theorem, for $n\geq 1$ the group $\K_{-n}(X)$ is a quotient of $\K_0(X\times\mathbf{G}_m^n)$ wherein elements coming from $\K_0(X\times\mathbf{A}^n)$ vanish.
\end{remark}

In order to define continous K-theory for adic rings, we give some reminders about adic rings and pro-objects.

\begin{reminder}
Let $A_0$ be be a ring and let $I$ be an ideal of $A_0$. Then the ideals $(I^n)_{n\geq 0}$ form a basis of neighbourhoods of zero in the so-called $I$-\df{adic topology}. An \df{adic ring} is a topological ring $A_0$ such that its topology coincides with the $I$-adic topology for some ideal $I$ of $A_0$. We say that $I$ is an \df{ideal of definition}. Note that adic rings have usually more than one ideal of definition. If the ideal $I$ is finitely generated, the completion $\hat{A}_0$ is naturally isomorphic to the limit $\lim_{n\geq 1}A_0/I^n$. 
\end{reminder}

\begin{reminder} \label{pro-objects--reminder}
We briefly recall the notion of pro-objects and, in particular, of pro-spectra. For proofs or references we refer to Kerz-Saito-Tamme \cite[\S 2]{kst18}.

Given an \infcat{} $\Ccal$ which is assumed to be accessible and to admit finite limits, one can built its \df{pro-category}
	\[
	\Pro(\Ccal) = \Fun^{\mathrm{lex,acc}}(\Ccal,\Spc)\op
	\] 
where $\Fun^{\mathrm{lex,acc}}(\Ccal,\Spc)$ is the full subcategory of $\Fun(\Ccal,\Spc)$ consisting of functors which are accessible (\ie{} preserve $\kappa$-small colimits for some regular cardinal number $\kappa$)  and left-exact (\ie{} commute with finite limits). The category $\Pro(\Ccal)$ has finite limits and, if $\Ccal$ has, also finite colimits which both can be computed level-wise. If $\Ccal$ is stable, then also $\Pro(\Ccal)$ is. 

As a matter of fact, a pro-object in $\Ccal$ can be represented by a functor $X\colon I\to\Ccal$ where $I$ is a small cofiltered \infcat{}. In this case, we write $\prolim_{i\in I}X_i$ for the corresponding object in $\Pro(\Ccal)$. In our situations, the index category $I$ will always be the poset of natural numbers $\N$.

Our main example of interest is the category $\Pro(\Sp)$ of \df{pro-spectra} whereas we are interested in another notion of equivalence. For this purpose, let $\iota \colon \Sp^+ \inj \Sp$ be the inclusion of the full stable subcategory spanned by bounded above spectra (\ie{} whose higher homotopy groups eventually vanish). The induced inclusion $\Pro(\iota) \colon \Pro(\Sp^+) \inj \Pro(\Sp)$ is right-adjoint to the restriction functor $\iota^* \colon \Pro(\Sp) \to \Pro(\Sp^+)$. 

A map $ X \to Y$ of pro-spectra is said to be a \df{weak equivalence} iff the induced map $\iota^*X\to\iota^*Y$ is an equivalence in $\Pro(\Sp^+)$. This nomenclature is justified by the fact that the map $X\to Y$ is a weak equivalence if and only if some truncation is an equivalence and the induced map on pro-homotopy groups are pro-isomorphisms. Similarly, one defines the notions of \df{weak fibre sequence} and \df{weak pullback}.
\end{reminder}

\begin{defin}
Let $A_0$ be a complete $I$-adic ring for some ideal $I$ of $A_0$. The \df{continuous K-theory} of $A_0$ is defined as the pro-spectrum
	\[
	\Kcont(A_0) = \prolim_{n\geq 1} \K(A_0/I^n)
	\]
where $\K$ is nonconnective algebraic K-theory (Definition~\ref{k-theory--def}). This is independent of the choice of the ideal of definition.
\end{defin}

\begin{defin} \label{tate-ring--defin}
A topological ring $A$ is called a \df{Tate ring} if there exists an open subring $\Anod \subset A$ which is a complete $\pi$-adic ring (\ie{} it is complete with respect to the $(\pi)$-adic topology) for some $\pi\in\Anod$ such that $A = \Anod[\pi\inv]$. We call such a subring $\Anod$ a \df{ring of definition} of A and such an element $\pi$ a \df{pseudo-uniformiser}. A \df{Tate pair} $(A,A_0)$ is a Tate ring together with the choice of a ring of definition and a \df{Tate triple} $(A,\Anod,\pi)$ is a Tate pair together with the choice of a pseudo-uniformiser.\footnote{One should not confuse our notion of a Tate pair with the notion of an \emph{affinoid Tate ring} $(A,A^+)$, \ie{} a Tate ring $A$ together with an open subring $A^+$ of the power-bounded elements of $A$ which is integrally closed in $A$. The latter one is used in the context of adic spaces.}
\end{defin}

\begin{example}
Every affinoid algebra $A$ over a complete \na{} field $k$ is a Tate ring. One can take $\Anod$ to be those elements $x$ which have residue norm $|x|_\alpha\leq 1$ with respect to some presentation $k\skp{t_1,\ldots,t_n} \surj[\alpha] A$ and any $\pi\in k$ with $|\pi|<1$ is a pseudo-uniformiser.
\end{example}

\begin{defin} \label{defin-continuous_k-theory}
Let $(A,\Anod,\pi) $ be a Tate triple. We define the \df{continuous K-theory} $\Kcont(A)$ of $A$ as the pushout
	\carre{\K(\Anod)}{\K(A)}{\Kcont(\Anod)}{\Kcont(A)}
in the $\infty$-category $\Pro(\Sp)$ of pro-spectra.
\end{defin}

\begin{remark} \label{cont-k-theory--rem}
In the situation of Definition~\ref{defin-continuous_k-theory} we obtain a fibre sequence
	\[
	\K(\Anod\on\pi) \To \Kcont(\Anod) \To \Kcont(A)
	\]
of pro-spectra. If $A = A^\prime_0[\lambda\inv]$ for another complete $\lambda$-adic ring $A^\prime_0$, one obtains a weakly equivalent pro-spectrum, \ie{} there is a zig-zag of maps inducing pro-isomorphisms on pro-homotopy groups \cite[Prop.~5.4]{kst18}. 
\end{remark}

For regular rings algebraic K-theory vanishes in negative degrees. For continuous K-theory this may not be the case since it sees the reduction type of a ring of definition.

\begin{example}
Let $(A,A_0)$ be a Tate pair. By definition there is an exact sequence 
	\[
	\ldots\to \K_{-1}(A_0) \to \Kcont_{-1}(A_0) \oplus \K_{-1}(A) \to \Kcont_{-1}(A) \to \K_{-2}(A_0) \to \ldots.
	\]
If both $A$ and $A_0$ are regular, it follows that $\Kcont_{-1}(A) \cong \K_{-1}(A_0)$. If $A_0$ is a $\pi$-adic ring, then $\Kcont_{-1}(A_0) = \K_{-1}(A_0/\pi)$ due to nilinvariance of negative algebraic K-theory (which follows from nilinvariance of $\K_0$ \cite[II. Lem.~2.2]{weibel} and the definition of negative K-theory in terms of $\K_0$ \cite[III. Def.~4.1]{weibel}). Now let $k$ be a discretely valued field and let $\pi\in k^\circ$ be a uniformiser.
\begin{enumerate}
	\item If $A \cong k\skp{x,y}/(x^3-y^2+\pi)$, we can choose $A_0 := k^\circ\skp{x,y}/(x^3-y^2+\pi)$ so that both $A$ and $A_0$ are regular. The reduction $A_0/\pi\cong \tilde{k}\skp{x,y}/(x^3-y^2)$ is the ``cusp'' over $\tilde{k}$. Thus $\Kcont_{-1}(A) = \K(A_0/\pi) = 0$ \cite[2.4]{weibel-normal}.
	\item If $A \cong k\skp{x,y}/(x^3+x^2-y^2+\pi)$, we can choose $A_0 := k^\circ\skp{x,y}/(x^3+x^2-y^2+\pi)$ so that both $A$ and $A_0$ are regular. The reduction $A_0/\pi \cong \tilde{k}\skp{x,y}/(x^3+x^2-y^2)$ is the ``node'' over $\tilde{k}$. If $\mathrm{char}(k)\neq 2$, then $\Kcont_{-1}(A) = \K(A_0/\pi) = \Z$ does not vanish \cite[2.4]{weibel-normal}.
\end{enumerate} 
\end{example}

For the reader's intuition we state some properties of continuous K-theory.

\begin{prop}[Kerz-Saito-Tamme]
Let $(A,A_0,\pi)$ be a Tate triple.
\begin{enumerate}
	\item The canonical map $\K_0(A) \to \Kcont_0(A)$ is an isomorphism.
	\item  $\Kcont_1(A) \cong \prolim\limits_n \K_1(A)/(1+\pi^nA_0)$.
	\item Continuous K-theory satifies an analytic version of Bass Fundamental Theorem; more precisely, for $i\in\Z$ there is an exact sequence
	\[
	0\to \Kcont_i(A) \to \Kcont_i(A\skp{t})\oplus\Kcont_i(A\skp{t\inv}) \to \Kcont_i(A\skp{t,t\inv}) \to \Kcont_{i-1}(A) \to 0.
	\]
	\item Continuous K-theory coincides in negative degrees with the groups defined by Karoubi-Villamayor \cite[7.7]{kv71}\footnote{Unfortunately, Karoubi-Villamayor call these groups ``positive''.} and Calvo \cite[3.2]{calvo}.
\end{enumerate}
\end{prop}
\begin{proof}
The statements (i), (iii), and (iv) are \cite[5.10]{kst18} and (ii) is \cite[5.5]{kst18}.
\end{proof}

There are not always rings of definition which behave nice enough so that we will have to deal with other models which may not be affine. Hence we define similarly to Definition~\ref{defin-continuous_k-theory} the following.

\begin{defin}
Let $X$ be a scheme over a $\pi$-adic ring $A_0$. Its \df{continuous K-theory} is
	\[
	\Kcont(X) := \prolim_{n\geq 1} K(X/\pi^n)
	\]
where $X/\pi^n := X \times_{\Spec(A_0)} \Spec(A_0/\pi^n)$.
\end{defin}

\begin{prop}[{Kerz-Saito-Tamme \cite[5.8]{kst18}}] \label{Kcont-model_fibre_seq--prop}
Let $(A,A_0,\pi)$ be a Tate triple such that $A_0$ is noetherian and let $X\to \Spec(A_0)$ be an admissible blow-up, \ie{} a proper morphism which is an isomorphism over $\Spec(A)$. Then there exists a weak fibre sequence
	\[
	\K(X \on\pi) \To \Kcont(X) \To \Kcont(A)
	\]
of pro-spectra.
\end{prop}

For a more detailed account of continuous K-theory we refer the reader to \cite[\S 6]{kst18}.

	\section{Admissible Zariski-Riemann spaces}
	\label{sec-ZR-spaces}

Using a regular model $X'$ of a regular affinoid algebra $A$ makes the fibre sequence (Proposition~\ref{Kcont-model_fibre_seq--prop})
	\[
	\K(X' \on\pi) \To \Kcont(X') \To \Kcont(A)
	\]
much easier as the left-hand term vanishes in negative degrees, \cf{} the proofsketch for the main result (p.~\pageref{proofsketch}). Unfortunately, resolution of singularities is not available at the moment in positive characteristic. 
A good workaround for this inconvenience is to work with a Zariski-Riemann type space which is defined as the inverse limit of all models, taken in the category of locally ringed spaces. This is not a scheme anymore, but behaves in the world of K-theory almost as good as a regular model does. For instance $\Kcont_n(A) \cong \Kcont_n(\skp{A_0}_A)$ for negative $n$ where $\skp{A_0}_A$ is the admissible Zariski-Riemann space associated with $A$ (Definition~\ref{admissible-blow-up--defin}).

The key part of this article is a comparison of rh-cohomology and Zariski cohomology for admissible Zariski-Riemann spaces (Theorem~\ref{Hzar=Hrh--thm}). Furthermore, we will see later that Zariski-Riemann spaces for formal schemes are closely related to adic spaces (Theorem~\ref{adic_space_is_ZR--thm}).

\begin{notation*}
In this section let $X$ be a \emph{reduced} quasi-compact and quasi-separated scheme and let $U$ be a quasi-compact open subscheme of $X$.
\end{notation*}

\begin{defin} \label{U-admissible_ZR-space--defin}
A \df{U-modification} of $X$ is a projective morphism $X'\to X$ of schemes which is an isomorphism over $U$. Denote by $\Mdf(X,U)$ the category of $U$-modifications of $X$ with morphisms over $X$. We define the \df{U-admissible Zariski-Riemann space} of $X$ to be the limit
	\[
	\skp{X}_U = \lim_{X'\in\Mdf(X,U)} X'
	\]
in the category of locally ringed spaces; it exists due to \cite[ch.~0, 4.1.10]{fuji-kato}.
\end{defin}

\begin{lemma} \label{ZR-coherent_sober--lem}
The underlying topological space of $\skp{X}_U$ is coherent and sober and for any $X'\in\Mdf(X,U)$ the projection $\skp{X}_U\to X'$ is quasi-compact.
\end{lemma}
\begin{proof}
This is a special case of \cite[ch.~0, 2.2.10]{fuji-kato}.
\end{proof}

The notion of a $U$-admissible modification is quite general. However, one can restrict to a more concrete notion, namely $U$-admissible blow-ups. 

\begin{defin}
A \df{U-admissible blow-up} is a blow-up $\Bl_Z(X) \to X$ whose centre $Z$ is finitely presented and contained in $X\setminus U$. Denote by $\BL(X,U)$ the category of $U$-admissible blow-ups with morphisms over $X$.
\end{defin}

\begin{prop}
The inclusion $\BL(X,U)\inj\Mdf(X,U)$ is cofinal. In particular, the canonical morphism
	\[
	\skp{X}_U = \lim_{X'\in \Mdf(X,U)} X' \To \lim_{X'\in \BL(X,U)} X'.
	\]
is an isomorphism of locally ringed spaces.
\end{prop}
\begin{proof}
Since a blow-up in a finitely presented centre is projective and an isomorphism outside its centre, $\BL(X,U)$ lies in $\Mdf(X,U)$. On the other hand, every $U$-modification is dominated by a $U$-admissible blow-up \cite[Lem.~2.1.5]{temkin08}. Hence the inclusion is cofinal and the limits agree.\footnote{\Cf{} the proof of Lemma~\ref{hcdh=cdh--lemma}.}
\end{proof}

\begin{lemma} \label{ZR-reduced-lem}
The full subcategory $\Mdf^\mathrm{\,\,red}(X,U)$ spanned by reduced schemes is cofinal in $\Mdf(X,U)$.
\end{lemma}
\begin{proof}
As $U$ is reduced by assumption, the map $X'\red\inj X'$ is a $U$-admissible blow-up for every $X'\in\Mdf(X,U)$.
\end{proof}

The remainder of this section is merely fixing notation for the application of admissible Zariski-Riemann spaces to the context of Tate rings.

\begin{defin} \label{admissible-blow-up--defin}
 Let $(A,\Anod,\pi)$ be a Tate triple (Definition~\ref{tate-ring--defin}). Setting $X=\Spec(\Anod)$ and $U=\Spec(A)$, we are in the situation of Definition~\ref{U-admissible_ZR-space--defin}. For simplicity we denote
	\[
	\Adm(\Anod) := \Mdf(\Spec(A),\Spec(\Anod))
	\]
and call its objects \df{admissible blow-ups}. Furthermore, we call the locally ringed space
	\[
	\skp{\Anod}_A := \skp{\Spec(\Anod)}_{\Spec(A)} = \lim_{X\in\Adm(\Anod)} X
	\]
the \df{admissible Zariski-Riemann space} associated to the pair $(A,\Anod)$.
\end{defin}

\begin{remark}
The admissible Zariski-Riemann space $\skp{A_0}_A$ depends on the choice of the ring of definition $\Anod$. However, if $B_0$ is another ring of definition, then also the intersection $C_0 := A_0 \cap B_0$ is. Hence we get a cospan 
	\[
	\Spec(A_0) \To \Spec(C_0) \oT \Spec(B_0)
	\]
which is compatible with the inclusions of $\Spec(A)$ into these. Hence every admissible blow-up $X \to \Spec(C_0)$ induces by pulling back an admissible blow-up $X_{A_0} \to \Spec(A_0)$ and a morphism $X \to X_{A_0}$. Precomposed with the canonical projections we obtain a map $\skp{A_0}_A \to X$. Hence the universal property yields a morphism $\skp{A_0}_A \to \skp{C_0}_A$. The same way, we get a morphism $\skp{B_0}_A \to \skp{C_0}_A$. One checks that the category of all admissible Zariski-Riemann spaces associated with $A$ is filtered.
\end{remark}

	\section{Cohomology of admissible Zariski-Riemann spaces}
	\label{sec-cohomology-ZR}

This section is the heart of this article providing the key ingredient for the proof of our main result; namely, a comparison of Zariski cohomology and rh-cohomology for admissible Zariski-Riemann spaces (Theorem~\ref{Hzar=Hrh--thm}). This will be done in two steps passing through the biZariski topology.

\begin{defin}
Let $S$ be a noetherian scheme. The \df{biZariski topology} is the topology generated by Zariski covers as well as by \df{closed covers}, \ie{} covers of the form $\{Z_i\to X\}_i$ where $X\in\Sch_S$ and the $Z_i$ are \emph{finitely many} jointly surjective closed subschemes of $X$. This yields a site $\Sch_S^\biZar$. 
\end{defin}

\begin{lemma}
The points on the biZariski site (in the sense of Goodwillie-Lichtenbaum \cite[\S 2]{goodwillie-lichtenbaum}) are precisely the spectra of integral local rings.
\end{lemma}
\begin{proof}
This follows from the fact that local rings are points for the Zariski topology and integral rings are points for the closed topology \cite{gabber-kelly}.
\end{proof}

\begin{lemma} \label{closed_cover_refinement--lem}
Let $X$ be a noetherian scheme. The cover of $X$ by its irreducible components refines every closed cover.
\end{lemma}
\begin{proof}
Let $(X_i)_i$ be the irreducible components of $X$ with generic points $\eta_i\in X_i$. Let $X = \bigcup_\alpha Z_\alpha$ be a closed cover. For every $i$ there exists an $\alpha$ such that $\eta_i\in Z_\alpha$, hence $X_i = \overline{\{\eta_i\}} \subseteq \overline Z_\alpha = Z_\alpha$. By maximality of the irreducible components we have equality.
\end{proof}

\begin{lemma} \label{constant-biZar_sheaf--lemma}
Let $S$ be a noetherian scheme. Every constant Zariski sheaf on $\Sch_S$ is already a bi-Zariski sheaf.
\end{lemma}
\begin{proof}
Let $A$ be an abelian group. For an open subset $U$ of $X\in\Sch_S$, the sections over $U$ are precisely the locally constant functions $f\colon U \to A$. By Lemma~\ref{closed_cover_refinement--lem}, it suffices to check the sheaf condition for the cover of $U$ by its irreducible components $(U_i)_i$. We only have to show the glueing property. If $f_i\colon U_i \to A$ are locally constant functions which agree on all intersections, then they glue to a function $f\colon U\to A$. We have to show that $f$ is locally constant. If $x\in U$, for every $i$ such that $x\in U_i$ there exists an open neighbourhood $V_i$ of $x$ in $U$ such that $f$ becomes constant when restricted to $U_i\cap V_i$. Hence $f$ becomes also constant when restricted to the intersection of all these $V_i$. Thus $f$ is locally constant.
\end{proof}

\begin{lemma} \label{Hzar=Hbizar--lemma}
Let $S$ be a noetherian scheme and let $X\in\Sch_S$. For any constant sheaf $A$ on $\Sch_S^\Zar$ we have $\Hzar^*(X;A) \cong \Hbizar^*(X;A)$.
\end{lemma}
\begin{proof}
Let $u\colon \Sch_S^\biZar \to \Sch_S^\Zar$ be the change of topology morphism of sites. Using the Leray spectral sequence
	\begin{align*}
		\Hzar^p(X;R^qu_*A) \Rightarrow \Hbizar^{p+q}(X;A)
	\end{align*}
it is enough to show that the higher images $R^qu_*A$ vanish for $q>0$. We know that $R^qu_*A$ is the Zariski sheaf associated with the presheaf
	\begin{align*}
		\Sch_S^\Zar \ni U \mapsto \Hbizar^q(U;A)
	\end{align*}
and that its stalks are given by $\Hbizar^q(X;A)$ for $X$ a local scheme (\ie{} the spectrum of a local ring). As the biZariski sheafification of $R^qu_*A$ is zero and using Lemma~\ref{closed_cover_refinement--lem}, we see that $\Hbizar^q(X;A)=0$ for every irreducible local scheme $X$. For a general local scheme $X$ we can reduce to the case where $X$ is covered by two irreducible components $Z_1$ and $Z_2$.

First, let $q=1$. We have an exact Mayer-Vietoris sequence
	\begin{align*} \tag{$\Delta$}
	0 \to A(X) \to A(Z_1) \times A(Z_2) \to[\alpha] A(Z_1\cap Z_2) \to[\partial] \\ \to[\partial] \Hbizar^1(X;A) \to \Hbizar^1(Z_1;A) \times \Hbizar^1(Z_2;A) \to \Hbizar^1(Z_1 \cap Z_2;A).
	\end{align*}
Since local schemes are connected, the map $\alpha$ is surjective, hence $\partial=0$ and the second line remains exact with a zero added on the left. Thus $\Hbizar^1(X;A)=0$ for any local scheme, hence $R^1u_*A$ vanishes.
For $q>1$ we proceed by induction. Let
	\begin{align*}
	0 \to A \to I \to G \to 0
	\end{align*}
be an exact sequence of biZariski sheaves such that  $I$ is injective. This yields a commutative diagram with exact rows and columns
	\[
	\begin{xy} 
	\xymatrix{
	I(Z_1) \times I(Z_2) \ar[r] \ar[d]
	& I(Z_1\cap Z_2) \ar[r] \ar[d]
	& 0		
	\\
	G(Z_1) \times G(Z_2) \ar[r]^-{\beta}
	& G(Z_1\cap Z_2)  \ar[d]
	\\
	& \Hbizar^1(Z_1\cap Z_2;A).
	}
	\end{xy}
	\]
Being a closed subscheme of a local scheme, $Z_1\cap Z_2$ is also a local scheme. By the case $q=1$, the group $\Hbizar^1(Z_1\cap Z_2;A)$ vanishes. Hence the map $\beta$ is surjective. Using the analogous Mayer-Vietoris sequence $(\Delta)$ above for $G$ instead of $A$, we can conclude that $R^2u_*A \cong R^1u_*G = 0$. Going on, we get the desired vanishing of $R^qu_*A$ for every $q>0$.
\end{proof}

The remainder of the sections deals with the rh-topology defined by Goodwillie-Lichtenbaum \cite{goodwillie-lichtenbaum}. We freely use results which are treated in a more detailed way in Appendix~\ref{sec-rh-topology}.

\begin{defin} \label{abs-rh--defin}
An \df{abstract blow-up square} is a cartesian diagram of schemes
	\[ \tag{abs}
	\xymatrix{
	E \ar[d] \ar[r] & \tilde X \ar[d] \\
	Z  \ar[r] & X
	}
	\]
where $Z\to X$ is a closed immersion, $\tilde X\to X$ is proper, and the induced morphism $\tilde X \setminus E \to X \setminus Z$ is an isomorphism. 
For any noetherian scheme $S$, the \df{rh-topology} on $\Sch_S$ is the topology generated by Zariski squares and covers $\{Z\to X, \tilde{X}\to X\}$ for every abstract blow-up square (abs) as well as the empty cover of the empty scheme.
\end{defin}

\begin{notation*}
\textbf{For the rest of this section, let $X$ be a reduced quasi-compact and quasi-separated scheme and let $U$ be a quasi-compact \emph{dense} open subscheme of $X$. We denote by $Z$ the closed complement equipped with the reduced scheme structure.}
\end{notation*}

\begin{defin} \label{admissible-blow-up2--defin}
For any morphism $p\colon X'\to X$ we get an analogous decomposition
	\[
	X'_Z \To X' \oT X'_U
	\]
where $X'_Z := X' \times_XZ$ and $X'_U := X'\times_XU$. By abuse of nomenclature, we call $X_Z$ the \df{special fibre} of $X'$ and $X_U$ the \df{generic fibre} of $X'$. An \df{(abstract) admissible blow-up} of $X'$ is a proper map $X''\to X'$ inducing an isomorphism $X''_U\iso X'_U$ over $X$. In particular, one obtains an abstract blow-up square
	\carre{X''_Z}{X''}{X'_Z}{X'.}
\end{defin}

At the end of this section, we will see that the Zariski cohomology and the rh-cohomology on the Zariski-Riemann space coincide for constant sheaves (Theorem \ref{Hzar=Hrh--thm}).
The following proposition will be used in the proof to reduce from the rh-topology to the biZariski topology. 

\begin{prop} \label{closed-cdh--prop}
Assume $X$ to be noetherian and let $X'\in\Sch_X$. Then for every proper rh-cover of the special fibre $X'_Z$ there exists an admissible blow-up $X'' \to X'$ such that the induced rh-cover of $ X''_Z$ can be refined by a closed cover.
\end{prop}
\begin{proof}
We may assume that $X'$ is reduced. Every proper rh-cover can be refined by a birational proper rh-cover (Lemma~\ref{generic-iso--lem}). Thus a cover yields a blow-up square which can be refined by an honest blow-up square
	\carre{E'}{Y'}{V'}{X'_Z}
\ie{} an abstract blow-up square where $Y'=\Bl_{V'}(X'_Z)$ (Lemma~\ref{hcdh=cdh--lemma}). We consider the honest blow-up square
	\carre{V''}{X'' := \Bl_{V'}(X')}{V'}{X'}
which is an admissible blow-up as $V'\subseteq X'_Z$ and decomposes into two cartesian squares
	\[ 
	\begin{xy} 
	\xymatrix{
	V'' \ar[r] \ar[d] & \tilde X''_Z \ar[r] \ar[d] & X'' \ar[d] \\
	V' \ar[r] & X'_Z \ar[r] & X'		
	}
	\end{xy}
	\]
where all the horizontal maps are closed immersions. By functoriality of blow-ups, we obtain a commutative diagram
	\carre{\Bl_{V''}(X''_Z)}{\Bl_{V''}(X'')=X''}{X''_Z}{X''}
wherein both horizontal maps are closed immersions and the right vertical map is an isomorphism by the universal property of the blow-up. Thus $\Bl_{V''}(X''_Z) \to X''_Z$ is a closed immersion \cite[Rem.~9.11]{gw10}. 
Functoriality of blow-ups yields a commutative square
	\carre{\Bl_{V''}(X''_Z)}{\Bl_{V'}(X'_Z)=Y'}{X''_Z}{X'_Z.}
By the universal property of the pullback, there exists a unique map $\Bl_{V''}(X''_Z) \to Y'' := Y' \times_{X'_Z} X''_Z$ such that following diagram commutes.
	\[
	\xymatrix@R=2.5ex@C=1.2ex{ 
	&&& \Bl_{V''}(X''_Z) \ar[dd] \ar[ddrr] \ar[dl] \\
	E'' \ar[rr] \ar[dd] \ar[dr] && Y'' \ar'[d][dd] \ar[dr] \\
	& V'' \ar[dd] \ar[rr] && X''_Z \ar[dd] \ar[rr] && X'' = \Bl_{V'}(X') \ar[dd] \\
	E' \ar[dr] \ar'[r][rr] && Y' \ar[dr] \\
	& V' \ar[rr] && X'_Z \ar[rr] && X'
	}
	\]
To sum up, we have shown that the pullback of the proper rh-cover $V'\sqcup Y'\to X'_Z$ along $X''_Z \to X'_Z$ can be refined by the closed cover $V'' \sqcup \Bl_{V''}(X''_Z) \to X''_Z$ which was to  be shown.
\end{proof}

Given a topology on (some appropriate subcategory of) the category of schemes, we want to have a corresponding topology on admissible Zariski-Riemann spaces. For this purpose, we will work with an appropriate site.

\begin{remark}
Let $\tau$ be a topology on the category $\Sch_X$. It restricts to a topology on the category $\Schqc_X$ of quasi-compact $X$-schemes. One obtains compatible topologies on the slice categories $\Schqc_{X'} = (\Schqc_X)_{/X'}$ for all $U$-modifications $X'\in\Mdf(X,U)$. 
\end{remark}

\begin{defin} \label{ZR-site--defin}
Consider the category
	\begin{align*}
		\Schqc(\skp{X}_U) := \colim_{X' \in \Mdf(X,U)} \Schqc_{X'}.
	\end{align*}
More precisely, the set of objects is the set of morphisms of schemes $Y'\to X'$ for some $X'\in\Mdf(X,U)$. The set of morphisms between two objects $Y'\to X'$ and $Y''\to X''$ is given by 
	\[
	\colim_{\tilde X} \, \Hom_\Sch(Y'\times_{X'}\tilde{X},Y''\times_{X''}\tilde{X})
	\]
where $\tilde{X}$ runs over all modifications $\tilde{X}\in\Mdf(X,U)$ dominating both $X'$ and $X''$.
Analogously, define the category 
	 \begin{align*}
	 \Schqc(\skp{X}_U\setminus U) := \colim_{X' \in \Mdf(X,U)} \Schqc_{X'\setminus U}
	 \end{align*}
where the $X'\setminus U$ are equipped with the reduced scheme structure.
\end{defin}

\begin{defin} \label{topology-zr--def}
Let $Y'\to X'$ be an object of $\Schqc(\skp{X}_U)$. We declare a sieve $R$ on $Y'$ to be a $\tau$\df{-covering sieve} of $Y'\to X'$ iff there exists a $U$-modification $p\colon X''\to X'$ such that the pullback sieve $p^*R$ lies in $\tau(Y'\times_{X'}X'')$. Analogously we define $\tau$-covering sieves in $\Schqc(\skp{X}_U\setminus U)$.
\end{defin}

\begin{lemma} \label{topology-zr--lem}
The collection of $\tau$-covering sieves in Definition~\ref{topology-zr--def} defines topologies on the categories $\Schqc(\skp{X}_U)$ and $\Schqc(\skp{X}_U\setminus U)$ which we will refer to with the same symbol $\tau$.
\end{lemma}
\begin{proof}
This follows immedeately from the construction.
\end{proof}
	
\begin{remark}
In practice, for working with the site $(\Schqc(\skp{X}_U),\tau)$ it is enough to consider $\tau$-covers in the category $\Schqc_X$ and identifying them with their pullbacks along $U$-modifications.
\end{remark}	

\begin{caveat}
The category $\Schqc(\skp{X}_U)$ is not a slice category, \ie{} a scheme $Y$ together with a morphism of locally ringed spaces $Y \to \skp{X}_U$ does not necessarily yield an object of $\Schqc(\skp{X}_U)$. Such objects were studied \eg{} by Hakim \cite{hakim}. In contrast, an object of $\Schqc(\skp{X}_U)$ is given by a scheme morphism $Y \to X'$ for some $X'\in\Mdf(X,U)$ and it is isomorphic to its pullbacks along admissible blow-ups.
\end{caveat}

In the proof of the main theorem we will need the following statement which follows from the construction of our site.

\begin{prop} \label{cohomology-zr--prop}
Let $F$ be a constant sheaf of abelian groups on $\Schqc(\skp{X}_U)$. Then the canonical morphism
	\[
	\colim_{X'\in\Mdf(X,U)} \Ho^*_\tau(X';F) \To \Ho^*_\tau(\skp{X}_U;F)
	\]
is an isomorphism. Analogously, if $F$ is a constant sheaf of abelian groups on $\Schqc(\skp{X}_U\setminus U)$, then the canonical morphism
	\[
	\colim_{X'\in\Mdf(X,U)} \Ho^*_\tau(X'_Z;F) \To \Ho^*_\tau(\skp{X}_U\setminus U;F)
	\]
is an isomorphism.
\end{prop}
\begin{proof}
This is a special case of \stacks{09YP} where the statement is given for any compatible system of abelian sheaves.
\end{proof}

\begin{thm} \label{Hzar=Hrh--thm}
For any constant sheaf $F$ on $\Schqc_\rh(\skp{X}_U)$, the canonical map
	\[
	\Hzar^*(\skp{X}_U\setminus U;F) \To \Hrh^*(\skp{X}_U\setminus U;F)
	\]
is an isomorphism.
\end{thm}
\begin{proof}
By construction, any rh-cover of $\skp{X}_U\setminus U$ is represented by an rh-cover of $X'_Z$ for some $X'\in\Adm(\Anod)$. We find a refinement $\tilde{V}\to[q]\tilde{Y}\to[p]X'_Z$ where $p$ is a proper rh-cover and $q$ is a Zariski cover (Proposition~\ref{rh-cover_factorisation--prop}). The rh-cover $\tilde{Y}\to X'_Z$ is given by $Y'\sqcup V'\to X$ for an abstract blow-up square
	\carre{E'}{Y'}{V'}{X'_Z.}
This is the situation of Proposition~\ref{closed-cdh--prop}. Thus there exists an admissible blow-up $X''\to X'$ and a refinement $V'' \sqcup \Bl_{V''}(X''_Z) \to X''_Z$ of the pulled back cover which consists of two closed immersions. Hence we have refined our given cover of $\skp{X}_U\setminus U$ by a composition of a Zariski cover and a closed cover which yields a bi-Zariski cover. This implies that $\Hrh^*(\skp{X}_U\setminus U;F)$ equals $\Hbizar(\skp{X}_U\setminus U;F)$. Now the assertion follows from Lemma~\ref{Hzar=Hbizar--lemma}.
\end{proof}                  

\begin{cor} \label{Hzar=Hrh-colimit_cor}
For any constant sheaf $F$, we have
	\begin{align*}
		\colim_{X' \in \Mdf(X,U)} \Hzar^*(X'_Z;F) = \colim_{X' \in \Mdf(X,U)} \Hrh^*(X'_Z;F).
	\end{align*}
\end{cor}
\begin{proof}
This is a formal consequence of the construction of the topology on $\Schqc(\skp{X}_U)$ since the cohomology of a limit site is the colimit of the cohomologies \stacks{09YP}.
\end{proof}

	\section{Formal Zariski-Riemann spaces and adic spaces}
	\label{sec-formal-zr-adic}
		
In this section we deal with Zariski-Riemann spaces which arise from formal schemes. According to a result of Scholze they are isomorphic to certain adic spaces (Theorem~\ref{adic_space_is_ZR--thm}). This identification is used in the proof of the main theorem (Theorem~\ref{main-thm}) to obtain the adic spectrum $\Spa(A,A^\circ)$ in the statement. We start with some preliminaries on formal schemes; for a detailled account of the subject we refer to Bosch's lecture notes \cite[pt.~II]{bosch}.

\begin{notation*}
In this section, let $R$ be a ring of one of the following types (\cf{} \cite[\S 7.3]{bosch}):
\begin{itemize}
	\item[(V)] $R$ is an adic valuation ring with finitely generated ideal of definition $I$.
	\item[(N)] $R$ is a noetherian adic ring with ideal of definition $I$ such that $R$ does not have $I$-torsion.
\end{itemize}
An $R$-algebra is called \emph{admissible} iff it is of topologically finite presentation and without $I$-torsion \cite[\S 7.3, Def.~3]{bosch}. A formal $R$-scheme is called \emph{admissible} iff it has a cover by affine formal $R$-schemes of the form $\Spf(A_0)$ for admissible $R$-algebras $A_0$, \cf{} \cite[\S 7.4, Def.~1]{bosch}. 
\end{notation*}

\begin{defin} \label{algebraic-formal--def}
For a scheme $X$ over $\Spec(R)$ we denote by $\hat{X}$ its associated formal scheme  $ \colim_n X/I^n$ over $\Spf(R)$. A formal scheme which is isomorphic to some $\hat{X}$ is called \df{algebraic}. Setting $U:=X\setminus X/I$, for every $U$-admissible blow-up $X' \to X$ the induced morphism of $\hat{X'} \to \hat{X}$ is an admissible formal blow-up \cite[3.1.3]{egr}. An admissible formal blow-up $\Xcal' \to \hat{X}$ of an algebraic formal scheme is called \df{algebraic} whenever it is induced from a $U$-admissible blow up of $X$.
\end{defin}

\begin{example}
\begin{enumerate}
 	\item Any quasi-affine formal scheme is algebraic. Indeed, an affine formal scheme $\Spf(\Anod)$ is isomorphic to the formal completion of $\Spec(\Anod)$. The quasi-affine case is Lemma~\ref{quasi-affine_formal_scheme_algebraic--lemma} below.	
 	\item For a \na{} field $k$, every projective rigid $k$-space has an algebraic model. In fact, any closed subspace of the rigid analytic space $\P_k^{n,\mathrm{an}}$ is the analytification of a closed subspace of $\P_k^n$ by a GAGA-type theorem \cite[4.10.5]{fresnel-vdp}. Since $\P_k^{n,\mathrm{an}}$ can be obtained by glueing $n+1$ closed unit discs $\Bb^n_k=\Spm(k\skp{t_1,\ldots,t_n})$ along algebraic maps \cite[4.3.4]{fresnel-vdp}, the rigid space $\P_k^{n,\mathrm{an}}$ is (isomorphic to) the generic fibre of the formal completion $(\P_{k^\circ}^n)^\wedge$ of the $k^\circ$-scheme $\P_{k^\circ}^n$; this argument also holds for closed subspaces. Hence every projective rigid $k$-space has an algebraic model.
\end{enumerate}
\end{example}

\begin{lemma} \label{algebraic_formal_blow-up--lemma}
Let $X$ be an $R$-scheme locally of finite type. Assume that $R$ is of type (N) or that $X$ is without $I$-torsion (\eg{} flat over $R$). Then every admissible formal blow-up of $\hat{X}$ is algebraic.
\end{lemma}
\begin{proof}
If $R$ is of type (N), then $R\skp{t_1,\ldots,t_n}$ is noetherian \cite[\S 7.3~ Rem.~1]{bosch} so that $\hat{X}$ is locally of topologically finite presentation. If $X$ is without $I$-torsion, then $\hat{X}$ is locally of topologically finite presentation \cite[\S 7.3, Cor.~5]{bosch}. Hence in both cases the notion of an admissible formal blow-up \cite[\S 8.2, Def.~3]{bosch} is defined. Set $X/I:=X\times_{\Spec(R)}\Spec(R/I)$ and let $\Ical$ be the ideal sheaf of $\O_X$ defining $X/I$.
Let $\Xcal'\to\hat{X}$  be an admissible formal blow-up defined by an open ideal $\Acal$ of $\O_{\hat{X}}$. In particular, there exists an $n\in\N$ such that $\Ical^n\O_{\hat{X}}\subset\Acal$. Let $Z_n := X/I^n$ be the closed subscheme of $X$ defined by $\Ical^n$. This yields a surjective map $\phi=i^\#\colon\O_X\to i_*\O_{Z_n}$ of sheaves on $X$ where $i \colon  Z_n\to X$ denotes the inclusion. Let $\tilde \Acal := \phi\inv\bigl( \Acal/(\Ical^n\O_{\hat{X}})\bigr)$. By construction, $i\inv\tilde\Acal=\Acal$ since both have the same pullback to $Z_k = (Z,\O_X/\Ical^k) = (Z,\O_{\hat{X}}/\Ical^k)$. Thus $\Xcal = \hat{X}_{\tilde\Acal}$.
\end{proof}

\begin{lemma} \label{Spf_basis--lem}
For every $R$-algebra $A_0$, the family  $\bigl(\Spf(A_0\skp{f\inv}\bigr)_{f\in \Anod}$ is a basis of the topology of $\Spf(A_0)$.
\end{lemma}
\begin{proof}
The family $\bigl((\Spec(A_0[f\inv])\bigr)_{f\in A_0}$ forms a basis of the topology of $\Spec(A_0)$. Topologically, $\Spf(A_0)$ is a closed subspace of $\Spec(A_0)$. Thus the induced family $\bigl(\Spec(A_0[f\inv]\cap\Spf(A_0)\bigr)_{f\in A_0}$ is a basis of the topology of $\Spf(A_0)$. As topological spaces,  $\Spf(A_0\skp{f\inv}) = \Spec(A_0[f\inv]) \cap \Spf(A_0)$. Hence we are done.
\end{proof}

\begin{lemma} \label{quasi-affine_formal_scheme_algebraic--lemma}
Every admissible formal blow-up of a quasi-affine admissible formal scheme is algebraic.
\end{lemma}
\begin{proof}
Let $j \colon \Ucal \inj \Xcal=\Spf(\Anod)$ be the inclusion of an open formal subscheme. and let $\Ucal' \to \Ucal$ be an admissible formal blow-up  defined by a coherent open ideal $\Acal_U\subseteq\O_\Ucal$.
Then there exists a coherent open ideal $\Acal\subseteq\O_\Xcal$ such that $\Acal|_U\cong\Acal_U$ and $\Acal|_V\cong \O_V$ whenever $V\cap U = \emptyset$ \cite[\S 8.2, Prop.~13]{bosch}. In particular, $\Ucal' \to \Ucal$ extends to an admissible formal blow-up $\Xcal' \to \Xcal$. 
By Lemma~\ref{algebraic_formal_blow-up--lemma}, this blow-up comes from an admissible blow-up $p \colon X' \to X=\Spec(A_0)$. 
By Lemma~\ref{Spf_basis--lem}, we can write $\Ucal = \bigcup_{i=1}^n \Ucal_i$ with $\Ucal_i = \Spf(A_0\skp{f_i\inv})$ for suitable $f_1,\ldots,f_n\in A_0$. Setting $U_i := \Spec(A_0[f_i\inv])$ and $U'_i := p\inv(U_i)$ and $U':= \bigcup_{i=1}^n U'_i$ the union in $X'$, then we obtain that 	\[
	\Ucal' = \bigcup_{i=1}^n \hat{U}_i = \hat{U'}
	\]
which finishes the proof.
\end{proof}

\begin{defin}
For a formal scheme $\Xcal$ locally of topologically finite presentation over $R$ its associated \df{formal Zariski-Riemann space} is defined to be the limit
	\[
	\skp{\Xcal} := \lim_{\Xcal'\in\Adm(\Xcal)} \Xcal_\Acal
	\]
in the category of locally topologically ringed spaces where $\Adm(\Xcal)$ denotes the category of all admissible formal blow-ups of $\Xcal$.
\end{defin}

\begin{lemma} \label{formal-algebraic-ZR--lemma}
Assume that the ideal ideal $I$ is principal, say generated by $\pi$. Let $X$ be an $R$-scheme locally of finite type. Assume that $R$ is of type (N) or that $X$ is without $\pi$-torsion (\eg{} flat over $R$). Then its formal completion $\hat{X}$ is homeomorphic to the special fibre $X/\pi = X\times_{\Spec(R)}\Spec(R/\pi)$. Consequently, the formal Zariski-Riemann space $\skp{\hat{X}}$ is homeomorphic to $\skp{X}_U/\pi = \skp{X}_U\setminus U$ where $U=\Spec(R[\pi\inv])$.
\end{lemma}
\begin{proof}
This is a direct consequence of the definition of a formal scheme \cite[\S 7.2]{bosch} and Lemma~\ref{algebraic_formal_blow-up--lemma}.
\end{proof}

\begin{thm}[{\cite[2.22]{scholze-perfectoid}}] \label{adic_space_is_ZR--thm}
Let $k$ be a complete \na{} field, \ie{} a topological field whose topology is induced by a \na{} norm, and let $k^\circ$ be its valuation ring.
Let $X\ad$ be a quasi-compact and quasi-separated adic space locally of finite type over $k$. Then there exists an admissible formal model $\Xcal$ of $X\ad$ and there is a homeomorphism $X\ad \to[\cong] \skp{\Xcal}$ which extends to an isomorphism
	\[
	(X\ad,\O_{X\ad}^+) \To[\cong] \lim_{\Xcal'\in\Adm(\Xcal)} (\Xcal',\O_{\Xcal'})
	\]
of locally ringed spaces.
\end{thm}

	\section{Main result: affinoid case}
	\label{sec-proof-main-result}

\begin{notation*}
In this section let $k$ be a complete discretely valued field with valuation ring $k^\circ$ and uniformiser $\pi$. This implies that the ring $k^\circ$ is noetherian.
\end{notation*}

\begin{thm} \label{main-thm}
Let $A$ be an affinoid $k$-algebra of dimension $d$. Then there is an isomorphism
	\[
	\Kcont_{-d}(A) \cong \Ho^d\bigl(\Spa(A,A^\circ);\Z\bigr)
	\]
where $\Spa(A,A^\circ)$ is the adic spectrum of $A$ with respect to its subring $A^\circ$ of power-bounded elements \cite[\S 3]{huber93} and the right-hand side is sheaf cohomology.
\end{thm}

Before proving the result, we first deduce an immediate consequence.

\begin{cor} \label{main-thm-cor}
Let $A$ be an affinoid $k$-algebra of dimension $d$. Then there is an isomorphism
	\[
	\Kcont_{-d}(A) \cong \Ho^d\bigl(\Spb(A);\Z\bigr)
	\]
where $\Spb(A)$ is the Berkovich spectrum of $A$ \cite[Ch.~1]{berkovich90} and the right-hand side is sheaf cohomology.
\end{cor}
\begin{proof}
The category of overconvergent\footnote{\Cf{} \cite[\S 4, p.~94]{put-schneider-points} for the definition of overconvergent sheaves.} sheaves on an adic spectrum is equivalent to the category of sheaves on the Berkovich spectrum \cite[\S 5, Thm.~6]{put-schneider-points}. The locally constant sheaf $\Z$ is overconvergent and admits a flasque resolution by overconvergent sheaves, hence the claim follows from Theorem~\ref{main-thm}.
\end{proof}

\begin{proof}[Proof of Theorem~\ref{main-thm}]
We may assume that $A$ is reduced as the statement is nilinvariant. Let $A^\circ$ be the subring of $A$ consisting of power-bounded elements of $A$. Then the pair $(A,A^\circ)$ is a Tate pair \cite[\S 6.2.4, Thm.~1]{bgr} and $A^\circ$ is noetherian \cite[\S 6.4.3, Prop.~3~(i)]{bgr}. For any $X\in\Adm(A^\circ)$ one has $X_A=\Spec(A)$ and thus by Proposition~\ref{Kcont-model_fibre_seq--prop} there is a fibre sequence
	\begin{align*}
		\K(X\on\pi) \To \Kcont(X) \To \Kcont(A).
	\end{align*}
Passing to the colimit over all admissible models we obtain a fibre sequence of pro-spectra
	\begin{align*}
	\colim_{X\in\Adm(A^\circ)}\K(X\on\pi) \To \colim_{X\in\Adm(A^\circ)}\Kcont(X) \To \Kcont(A).
	\end{align*}
For $i<0$ we have that $\colim_{X\in\Adm(A^\circ)}\K_i(X\on\pi)=0$ \cite[Prop.~7]{kerz-icm} and hence
	\begin{align*}
	\Kcont_i(A) \cong \colim_{X\in\Adm(A^\circ)}\Kcont_i(X).
	\end{align*}	  
Lemma~\ref{K_-d_is_nil-invariant--lemma} below and Theorem~\ref{K_-d=Hrh^d--thm} yield 
	\begin{align*}
	\colim_{X\in\Adm(A^\circ)}\Kcont_i(X)
	&\cong \colim_{X\in\Adm(A^\circ)}\K_{-d}(X/\pi)\\
	&\cong \colim_{X\in\Adm(A^\circ)}\K_{-d}(X/\pi) \\
	&\cong \colim_{X\in\Adm(A^\circ)}\Hrh^d(X/\pi;\Z)
	\end{align*}
where the last isomorphism uses that $d=\dim(X/\pi)$ if $X\in\Adm(A^\circ)$ is reduced.
Corollary~\ref{Hzar=Hrh-colimit_cor} says that
	\[
	\colim_{X\in\Adm(A^\circ)}\Hrh^*(X/\pi;\Z) \cong \colim_{X\in\Adm(A^\circ)}\Hzar^*(X/\pi;\Z).
	\]
The Zariski cohomology is just ordinary sheaf cohomology. The latter one commutes with colimits of coherent and sober spaces with quasi-compact transition maps \cite[ch.~0, 4.4.1]{fuji-kato}. Since the admissible Zariski-Riemann space is such a colimit we obtain 
	\[
	\colim_{X\in\Adm(A^\circ)}\Hzar^*(X/\pi;\Z) \cong \Ho^*(\skp{A^\circ}_A/\pi;\Z).
	\]
where the right-hand side is sheaf cohomology. Finally we get that
	\[
	\Ho^*\bigl(\skp{A^\circ}_A/\pi;\Z\bigr) \cong \Ho^*\bigl(\Spa(A,A^\circ);\Z\bigr).
	\]
since the admissible Zariski-Riemann space $\skp{A^\circ}_A$ is homeomorphic to the formal Zariski-Riemann space $\skp{\Spf(A^\circ)}$ (Lemma~\ref{formal-algebraic-ZR--lemma}) which is isomorphic to the adic spectrum $\Spa(A,A^\circ)$ (Theorem~\ref{adic_space_is_ZR--thm}).
\end{proof}

\begin{lemma} \label{K_-d_is_nil-invariant--lemma}
Let $Y$ be a noetherian scheme of finite dimension $d$. Then for $n\geq d$ we have
	\begin{align*}
	 \K_{-n}(Y) \cong \K_{-n}(Y\red).
	\end{align*}
\end{lemma}
\begin{proof}
This follows by using the Zariski-descent spectral sequence and nilinvariance of negative algebraic K-theory for affine schemes. 
\end{proof}

	\section{Continous K-theory for rigid spaces}	
	\label{sec-global-continuous-k-theory}

In this section we see that continuous K-theory, as defined for algebras in Definition~\ref{defin-continuous_k-theory}, satisfies descent and hence defines a sheaf of pro-spectra for the admissible topology. The result and its proof are due to Morrow \cite{morrow}; we present here a slightly different argument. For the general theory on rigid $k$-spaces we refer the reader to Bosch's lecture notes \cite[pt.~I]{bosch}.

\begin{notation*}
In this section let $k$ be a complete discretely valued field with valuation ring $k^\circ$ and uniformiser $\pi$. This implies that the ring $k^\circ$ is noetherian. For an affinoid $k$-algebra $A$ denote by $\Spm(A)$ its associated affinoid $k$-space \cite[\S 3.2]{bosch}.\footnote{Bosch uses the notation $\mathrm{Sp}(A)$.}
Denote by $\FSch_{k^\circ}$ the category of formal schemes over $k^\circ$ and by $\FSch_{k^\circ}^\mathrm{lft}$ its full subcategory of formal schemes that are locally finite type over $k^\circ$; we consider these as sites equipped with the Zariski-topology.
\end{notation*}

\begin{lemma} \label{Kcont_formal_Zariski_descent--lemma}
Let $\Xcal$ be a formal scheme over $k^\circ$ which is assumed to be covered by two open formal subschemes $\Xcal_1$ and $\Xcal_2$. Setting $\Xcal_3 := \Xcal_1\cap\Xcal_2$ we obtain a cartesian square
	\carre{\Kcont(\Xcal)}{\Kcont(\Xcal_1)}{\Kcont(\Xcal_2)}{\Kcont(\Xcal_3)}
in the category $\Pro(\Sp)$.
\end{lemma}
\begin{proof}
For every $n\geq 1$, the special fibre $\Xcal/\pi^n$ is covered by $\Xcal_1/\pi^n$ and $\Xcal_2/\pi^n$ with intersection $\Xcal_3/\pi^n$. Applying algebraic K-theory one obtains cartesian squares by Zariski descent. Now the claim follows as finite limits in the pro-category can be computed levelwise (Reminder~\ref{pro-objects--reminder}).
\end{proof}

\begin{cor} \label{Kcont_formal_sheaf--cor}
The presheaf $\Kcont$ on the site $\FSch_{k^\circ}$ is a sheaf of pro-spectra and satisfies $\Kcont(\Spf(A_0)) \simeq \Kcont(\Anod)$ for every $k^\circ$-algebra  $\Anod$.
\end{cor}
\begin{proof}
This is a standard consequence for topologies which are induced by cd-structures \cite[Thm. 3.2.5]{ahw1}.
\end{proof}

\begin{lemma}[{\cite[3.4]{morrow}}]
Let $\Spm(A)$ be an affinoid $k$-space which is assumed to be covered by two open affinoid subdomains $\Spm(A^1)$ and $\Spm(A^2)$. We set
	\[ 
	 A^3 := A^1 \hat{\otimes}_A A^2 := k \otimes_{k_0} (A_+^1 \otimes_{A_+} A_+^2)^\wedge
	 \]
where $A_0, A_+^1, A_+^2$ are respective subrings of definition of $A, A^1, A^2$ and $(A_+^1 \otimes_{A_0} A_+^2)^\wedge$ denotes the $\pi$-adic completion. Then the square
	\carretag[$\square$]{\Kcont(A)}{\Kcont(A^1)}{\Kcont(A^2)}{\Kcont(A^3)}
is weakly cartesian in $\Pro(\Sp)$, \ie{} cartesian in $\Pro(\Sp^+)$.
\end{lemma}
\begin{proof}
We note that the definition of the ring $A^3$ is independent of the choices of the rings of definition $A_+^1$ and $A_+^2$ and we forget about these choices.
According to Raynaud's equivalence of categories between quasi-compact admissible formal $k^\circ$-schemes localised by admissible formal blow-ups and quasi-compact and quasi-separated rigid $k$-spaces we find an admissible formal blow-up $\Xcal \to \Spf(A_0)$ and an open cover $\Xcal = \Xcal_1 \cup \Xcal_2$ whose associated generic fibre is the given cover $\Spm(A) = \Spm(A^1) \cup \Spm(A^2)$ \cite[\S 8.4]{bosch}. Since every admissible blow-up of the algebraic formal scheme $\Spf(A_0)$ is algebraic (Lemma~\ref{algebraic_formal_blow-up--lemma}), we find an admissible blow-up $X\to \Spec(A_0)$ and an open cover $X=X_1 \cup X_2$ whose formal completion is the cover $\Xcal = \Xcal_1 \cup \Xcal_2$. We set $X_3 := X_1 \cap X_2$ and note that for $i\in\{1,2,3\}$ there exist rings of definition $A^i_0$ of $A^i$, respectively, and admissible blow-ups $X_i\to\Spec(A^i_0)$. By Zariski descent have two cartesian squares
\[ \begin{xy} \xymatrix{
	\K(X\on\pi) \ar[d] \ar[r] & \K(X_1\on\pi) \ar[d]
	&& \Kcont(X) \ar[d] \ar[r] & \Kcont(X_1) \ar[d]
	\\
	\K(X_2\on\pi) \ar[r] & \K(X_3\on\pi)
	&& \Kcont(X_2) \ar[r] & \Kcont(X_3)	
} \end{xy} \]
where the right square is cartesian since it is levelwise cartesian (Reminder~\ref{pro-objects--reminder}). There is map from the left square to the right square. By Proposition~\ref{Kcont-model_fibre_seq--prop}, the square of cofibres is weakly equivalent to the square ($\square$) which is therefore weakly cartesian.
\end{proof}

The following statement is a standard result about extending sheaves from local objects to global ones and permits us to extend continuous K-theory to the category of rigid $k$-spaces. 

\begin{prop}
The inclusion $\iota \colon \Rig_k^\mathrm{aff} \inj \Rig_k$ of affinoid $k$-spaces into rigid $k$-spaces induces an equivalence
	\[
	\iota^* \Colon \Sh(\Rig_k) \To[\simeq] \Sh(\Rig_k^\mathrm{aff}).
	\]
Moreover, for every \infcat{} $\Dcal$ which admits small limits, the canonical map
	\[
	\iota^* \Colon \Sh_\Dcal(\Rig_k) \To[\simeq] \Sh_\Dcal(\Rig_k^\mathrm{aff}).
	\]
is an equivalence.
\end{prop}
\begin{proof}
This follows from applying twice an $\infty$-categorical version of the ``comparison lemma'' \cite[C.3]{hoyois-quadratic}: first to the inclusion $\Rig_k^\mathrm{aff} \inj \Rig_k^\mathrm{sep}$ of affinoid spaces into separated spaces and secondly to the inclusion $\Rig_k^\mathrm{sep} \inj \Rig_k$.
\end{proof}

\begin{cor}[{\cite[3.5]{morrow}}] \label{Kcont_rigid_sheaf--cor}
There exists a unique sheaf $\Kcont$ on the category $\Rig_k$ (equipped with the admissible topology) that has values in $\Pro(\Sp^+)$ and satisfies $\Kcont(\Spm(A)) \simeq \Kcont(A)$ for every affinoid $k$-algebra $A$.
\end{cor}

\begin{cor} \label{Kcont_eta_formal_sheaf--cor}
The functor
	\[
	\Kcont((\_)_\eta) \Colon (\FSch_{k^\circ})\op \to (\Rig_k)\op \to  \Pro(\Sp^+), \quad \Xcal \mapsto \Xcal_\eta \mapsto \Kcont(\Xcal_\eta)
	\]
is a sheaf.
\end{cor}
\begin{proof}
This follows from the fact that Zariski covers of formal schemes induce on generic fibres admissible covers of rigid spaces.
\end{proof}

	\section{Main Result: global case}	
	\label{sec-global-main-result}

In this section we conjecture that an analogous version of our main result (Theorem~\ref{main-thm}) for rigid spaces is true. We prove this conjecture in the algebraic case (\eg{} affinoid or projective) and in dimension at least two (Theorem~\ref{global-main-thm}). The constructions in this section are ad-hoc for our purposes and a full development of the formalism which will be based on adic spaces needs to be examined in future work.

\begin{notation*}
In this section let $k$ be a complete discretely valued field with valuation ring $k^\circ$ and uniformiser $\pi$. 
\end{notation*}

\begin{conj} \label{global-main-thm--conj}
Let $X$ be a quasi-compact and quasi-separated rigid $k$-space of dimension $d$. Then there is an isomorphism
	\[
	\Kcont_{-d}(X) \cong \Ho^d(X;\Z)
	\]
of pro-abelian groups. In particular, the pro-abelian group $\Kcont_{-d}(X)$ is constant.
\end{conj}

\begin{defin} \label{Kcont_on_pi--defin}
For an affine formal scheme $\Spf(\Anod)$ with associated generic fibre $\Spm(A)$ where $A=\Anod\otimes_{k^\circ}k$ there is by definition a map $\Kcont(\Anod)\to\Kcont(A)$. This map can be seen as a natural transformation $\FSch\aff\to\Pro(\Sp^+)$ which extends to a natural transformation
	\[
	\Kcont(\_) \to \Kcont((\_)_\eta) \Colon (\FSch_{k^\circ}^{\mathrm{lft}})\op \To \Pro(\Sp^+).
	\]
For a formal scheme $\Xcal$ locally of finite type over $k^\circ$ we define
	\[
	\Kcont(\Xcal\on\pi) \,:=\, \fib\bigl( \Kcont(\Xcal) \to \Kcont(\Xcal_\eta) \bigr)
	\]
where $\Xcal_\eta$ is the associated generic fibre. By construction and by Corollary~\ref{Kcont_formal_sheaf--cor} and Corollary~\ref{Kcont_eta_formal_sheaf--cor} the induced functor 
	\[
	\Kcont(\_\on\pi) \Colon (\FSch_{k^\circ}^{\mathrm{lft}})\op \to \Pro(\Sp^+)
	\]
is a sheaf.
\end{defin}

\begin{lemma} \label{K_on_pi=K^cont_on_pi--lemma}
Let $X$ be a $k^\circ$-scheme locally of finite type. Then there is a canonical equivalence
	\[
	\K(X\on\pi) \To[\simeq] \Kcont(\hat{X}\on\pi).
	\]
In particular, $\Kcont(\hat{X}\on\pi)$ is equivalent to a constant pro-spectrum.
\end{lemma}
\begin{proof}
If $X=\Spec(\Anod)$ is affine we have by Definition~\ref{defin-continuous_k-theory} a pushout square
	\carre{\K(\Anod)}{\K(\Anod[\pi\inv])}{\Kcont(\Anod)}{\Kcont(\Anod[\pi\inv]).}
Since the category $\Sp$ is stable, this also holds for $\Pro(\Sp)$. Thus the square is also a pullback and we have an equivalence $\K(\Anod\on\pi)\simeq\Kcont(\Anod\on\pi)$ of the horizontal fibres.
If $X$ is quasi-compact and separated, choose a finite affine cover $(U_i)_i$ which yields a commutative diagram
	\carre{\K(X\on\pi)}{\lim\limits_\Delta \K(\check{U}_\bullet\on\pi)}{\Kcont(\hat{X}\on\pi)}{\lim\limits_\Delta \Kcont(\check{\hat{U}}_\bullet\on\pi)}
where $\check{U}_\bullet$ and $\check{\hat{U}}_\bullet$ are the $\check{\mathrm{C}}$ech nerves of the cover $(U_i)_i$ of $X$ respectively the induced cover $(\hat{U}_i)_i$ of $\hat{X}$. Thus the horizontal maps are equivalences. By the affine case the right vertical map is an equivalence, hence also the left vertical map as desired. For the quasi-separated case we reduce analogously to the separated case.
\end{proof}

\begin{cor}
Let $X$ be a $k^\circ$-scheme locally of finite type. Then the square
	\carre{\K(X)}{\K(X_k)}{\Kcont(\hat{X})}{\Kcont(\hat{X}_\eta)}
is cartesian in $\Pro(\Sp^+)$ where $X_k := X \times_{\Spec(k^\circ)} \Spec(k)$.	
\end{cor}
\begin{proof}
By design there is a commutative diagram of fibre sequences
	\[
	\begin{xy} 
	\xymatrix{
	\K(X\on\pi) \ar[r] \ar[d] & \K(X) \ar[r] \ar[d] & \K(X_k) \ar[d]\\
	\Kcont(\hat{X}\on\pi) \ar[r] & \Kcont(\hat{X})	 \ar[r] & \Kcont(\hat{X}_\eta)
	}
	\end{xy}
	\]
where the left vertical map is an equivalence due to Lemma~\ref{K_on_pi=K^cont_on_pi--lemma}. 
\end{proof}

\begin{cor} \label{algebraic_colimit_vanishing--cor}
Let $X$ be a reduced $k^\circ$-scheme locally of finite type. For $n\geq 1$ we have 
	\[
	\colim_{\Xcal'} \, \Kcont_{-n}(\Xcal'\on\pi) = 0
	\]
where $\Xcal'$ runs over all admissible formal blow-ups of $\hat{X}$.\footnote{There is no trouble with this colimit since the pro-spectrum in question is constant by Lemma~\ref{K_on_pi=K^cont_on_pi--lemma}. In general, colimits in pro-categories are hard to compute.}
\end{cor}
\begin{proof}
Since every admissible formal blow-up of an algebraic formal scheme is algebraic (Lemma~\ref{algebraic_formal_blow-up--lemma}), due to Lemma~\ref{K_on_pi=K^cont_on_pi--lemma}, and by \cite[Prop.~7]{kerz-icm} we have  
	\[
	\colim_{\Xcal'}\, \Kcont_{-n}(\Xcal'\on\pi) \cong \colim_{X'}\, \Kcont_{-n}(\hat{X'}\on\pi) \cong \colim_{X'}\, \K_{-n}(X'\on\pi) = 0 
	\]
where the latter two colimits are indexed by all $X_k$-admissible blow-ups of and where $X_k := X \times_{\Spec(k^\circ)} \Spec(k)$.	
\end{proof}

\begin{lemma} \label{formal_colimit_vanishing--lemma}
Let $\Xcal$ be a quasi-compact admissible formal scheme. For $n\geq 2$ have 
	\[
	\colim_{\Xcal'} \, \Kcont_{-n}(\Xcal'\on\pi) = 0
	\]
where $\Xcal'$ runs over all admissible formal blow-ups of $\Xcal$.
\end{lemma}
\begin{proof}
Let $\alpha\in\Kcont_{-n}(\Xcal\on\pi)$. We choose a finite affine cover $(\Ucal_i)_{i\in I}$ of $\Xcal$ where $I=\{1,\ldots,k\}$. By the affine case, we find for every $i\in I$ an admissible formal blow-up $\Ucal'_i\to\Ucal_i$ such that the map $\Kcont_{-n}(\Ucal_i)\to\Kcont_{-n}(\Ucal'_i)$ sends $\alpha|_{\Ucal_i}$ to zero. There exists an admissible formal blow-up $\Xcal'\to\Xcal$ locally dominating these local blow-ups, \ie{} for every $i\in I$ the pullback $\Xcal'\times_{\Xcal}\Ucal_i \to \Ucal_i$ factors over $\Ucal'\to\Ucal$ \cite[8.2, Prop.~14]{bosch}. We may assume that $\Xcal'\times_\Xcal\Ucal_i = \Ucal'_i$. Setting $\Vcal := \Ucal_2 \cup \ldots \cup \Ucal_k$ one obtains a commutative diagram
	\[
	\begin{xy} 
	\xymatrix{
	\Kcont_{-n+1}(\Ucal_1\cap\Vcal\on\pi) \ar[r] \ar[d]
	& \Kcont_{-n}(\Xcal\on\pi) \ar[r] \ar[d]
	& \Kcont_{-n}(\Ucal_1\on\pi) \oplus \Kcont_{-n}(\Vcal\on\pi) \ar[d]
	\\
	\Kcont_{-n+1}(\Ucal'_1\cap\Vcal'\on\pi) \ar[r] 
	& \Kcont_{-n}(\Xcal'\on\pi) \ar[r] 
	& \Kcont_{-n}(\Ucal'_1\on\pi) \oplus \Kcont_{-n}(\Vcal'\on\pi)		
	}
	\end{xy}
	\]
of Mayer-Vietoris sequences. By the affine case and by induction on the cardinality of the affine cover, $\alpha$ maps to zero in $\Kcont_{-n}(\Ucal'_1) \oplus \Kcont_{-n}(\Vcal')$. Hence its image in $\Kcont_{-n}(\Xcal'\on\pi)$ comes from an element $\alpha'$ in $\Kcont_{-n+1}(\Ucal'_1\cap\Vcal')$. As an admissible formal blow-up of the quasi-affine admissible formal scheme $\Ucal_1\cap\Vcal$, the formal scheme $\Ucal_1'\cap\Vcal'$ is algebraic according to Lemma~\ref{quasi-affine_formal_scheme_algebraic--lemma}. Hence there exists an admissible formal blow-up of $\Ucal_1'\cap\Vcal'$ where $\alpha'$ vanishes. As above this can be dominated by an admissible formal blow-up $\Xcal''\to\Xcal'$ so that the image of $\alpha$ in $\Kcont_{-n}(\Xcal''\on\pi)$ vanishes.
\end{proof}

Next we do another similar reduction to the affinoid case.

\begin{lemma} \label{Hzar=Hrh-formal--lem}
For every quasi-compact admissible formal scheme $\Xcal$ and every constant rh-sheaf $F$ the canonical map
	\[
	\colim_{\Xcal'}\, \Hzar^*(\Xcal'/\pi;F) \To \colim_{\Xcal'}\, \Hrh^*(\Xcal'/\pi;F)
	\]
is an isomorphism.
\end{lemma}
\begin{proof}
This is similar to the proof of Proposition~\ref{formal_colimit_vanishing--lemma}. By a Mayer-Vietoris argument and by induction on the number of affine formal schemes needed to cover $\Xcal$, we can reduce to one degree less. Fortunately, this also works in degree $0$ due to the sheaf condition.
\end{proof}

We now prove Conjecture~\ref{global-main-thm--conj} in almost all cases.

\begin{thm} \label{global-main-thm}
Let $X$ be a quasi-compact and quasi-separated rigid $k$-space of dimension $d$. Assume that $d\geq 2$ or that there exists a formal model which is algebraic (Definition~\ref{algebraic-formal--def}, \eg{} $X$ is affinoid or projective). Then there is an isomorphism 
	\[
	\Kcont_{-d}(X) \cong \Ho^d(X;\Z)
	\]
where the right-hand side is sheaf cohomology with respect to the admissible topology on the category of rigid $k$-spaces.
\end{thm}
\begin{proof}
Let $\Xcal$ be an admissible formal model of $X$. By Definition~\ref{Kcont_on_pi--defin} there is a fibre sequence 
	\[
	\Kcont(\Xcal\on\pi) \To \Kcont(\Xcal) \To \Kcont(X).
	\]
If $\colim_{\Xcal'} \Kcont_{-n}(\Xcal'\on\pi) = 0$ for $n\in\{d-1,d\}$, then the induced map
	\[
	\colim_{\Xcal'} \,\Kcont_{-d}(\Xcal') \To \Kcont_{-d}(X)
	\]
is an isomorphism; this is the case if $\Xcal$ is algebraic (Corollary~\ref{algebraic_colimit_vanishing--cor}; note that the reducedness assumption does not harm due to Lemma~\ref{K_-d_is_nil-invariant--lemma}) or if $\dim(X)\geq 2$ (Lemma~\ref{formal_colimit_vanishing--lemma}). By Lemma~\ref{K_-d_is_nil-invariant--lemma} and Theorem~\ref{K_-d=Hrh^d--thm} we conclude
	\[
	\Kcont_{-d}(X) \cong \colim_{\Xcal'}\,\Kcont_{-d}(\Xcal) \cong \colim_{\Xcal'}\,\K_{-d}(\Xcal/\pi) \cong \colim_{\Xcal'}\,\Hrh^d(\Xcal/\pi;\Z).
	\]
By Lemma~\ref{Hzar=Hrh-formal--lem} we have that
	\[
	\colim_{\Xcal'}\,\Hrh^*(\Xcal/\pi;\Z) \cong \colim_{\Xcal'}\, \Hzar^*(\Xcal'/\pi;\Z).
	\]
Since every formal scheme is homeomorphic to its special fibre, the latter one identifies with $\colim_{\Xcal'}\,\Ho^d(\Xcal';\Z)$ as sheaf cohomology only depends on the topology. Since the formal Zariski-Riemann space $\skp{\Xcal} = \lim_{\Xcal'}\Xcal'$  is a colimit of coherent and sober spaces with quasi-compact transition maps, it commutes with cohomology \cite[ch.~0, 4.4.1]{fuji-kato}. Hence we conclude that
	\[
	\colim_{\Xcal'}\,\Ho^d(\Xcal';\Z) \cong \Ho^d(\skp{\Xcal};\Z)) \cong \Ho^d(X\ad;\Z) \cong \Ho^d(X;\Z)
	\]
by using Theorem~\ref{adic_space_is_ZR--thm} for the middle isomorphism.
\end{proof}

\begin{remark}
The cases of Conjecture~\ref{global-main-thm--conj} which are not covered by Theorem~\ref{global-main-thm} are curves which are not algebraic. In particular, they must not be affine nor projective nor smooth proper (\cf{} \cite[1.8.1]{luetkebohmert-curves}).
\end{remark}

As the constant sheaf $\Z$ is overconvergent we infer the following.

\begin{cor} \label{global-main-thm-cor}
Let $X$ be a quasi-compact and quasi-separated rigid analytic space of dimension $d$ over a discretely valued field. Assume that $d\geq 2$ or that there exists a formal model of $X$ which is algebraic (\eg{} $X$ is affinoid or projective). Then there is an isomorphism 
	\[
	\Kcont_{-d}(X) \cong \Ho^d(X\berk;\Z)
	\]
where $X\berk$ is the Berkovich space associated with $X$. 
\end{cor}

\begin{remark}
If $X\berk$ is smooth over $k$ or the completion of a $k$-scheme of finite type, then there is an isomorphism \cite[III.1.1]{bredon-sheaf}
	\[
	\Ho^d(X\berk;\Z) \cong \Hsing^d(X\berk;\Z)
	\]
with singular cohomology since the Berkovich space $X\berk$ is locally contractible. For smooth Berkovich spaces this is a result of Berkovich \cite[9.1]{berkovich99} and for completions of $k$-schemes of finite type this was proven by Hrushovski-Loeser \cite{hrushovski-loeser}.
\end{remark}

Finally, we prove vanishing and homotopy invariance of continous K-theory in low degrees. The corresponding statement for affinoid algebras was proven by Kerz \cite[Thm.~12]{kerz-icm}.

\begin{thm} \label{Kcont-vanishing--thm}
Let $k$ be a complete discretely valued field and let $X$ be a quasi-compact and quasi-separated rigid $k$-space of dimension $d$. Then:
\begin{enumerate}
\item $\Kcont_{-i}(X)=0$ for $i>d$.
\item The canonical map $\Kcont_{-i}(X) \to \Kcont_{-i}(X\times\mathbf{B}_k^n)$ is an isomorphism for $i\geq d$ and $n\geq 1$ where $\mathbf{B}_k^n:=\Spm(k\skp{t_1,\ldots,t_n})$ is the rigid unit disc.
\end{enumerate}
\end{thm}
\begin{proof}
Let $i\geq 1$. We have an exact sequence
\[
\colim_{\Xcal}\Kcont_{-i}(\Xcal\on\pi) \to \colim_{\Xcal}\Kcont_{-i}(\Xcal) \to \Kcont_{-i}(X) \to \colim_{\Xcal}\Kcont_{-i-1}(\Xcal\on\pi)
\] 
where $\Xcal$ runs over all admissible formal models of $X$. The last term in the sequence vanishes due to Lemma~\ref{formal_colimit_vanishing--lemma}. For $i>d$, we have $\Kcont_{-i}(\Xcal) = \K_{-i}(\Xcal/\pi) = 0$ by Lemma~\ref{K_-d_is_nil-invariant--lemma} and vanishing of algebraic K-theory as $\dim(\Xcal/\pi) = d$. This shows (i).
For (ii) this works analogously for $N^\mathrm{cont}_{-i,n}(X) := \coker\bigl( \Kcont_{-i}(X) \to \Kcont_{-i}(X\times\mathbf{B}_k^n) \bigr)$ which is enough as the map in question is injective.
\end{proof}

	\appendix	
	\section{The rh-topology}
	\label{sec-rh-topology}

In this section we examine the rh-topology introduced by Goodwillie-Lichtenbaum \cite[1.2]{goodwillie-lichtenbaum}. We use a different definition in terms of abstract blow-up squares and show that both definitions agree (Corollary~\ref{rh-covers-agree--cor}). In the end, we will prove some rh-versions of known results for the cdh-topology; most importantly, that the rh-sheafification of K-theory is KH-theory (Theorem~\ref{Lrh_K=KH--thm}).

\begin{notation*}
Every scheme in this section is noetherian of finite dimension. Under these circumstances, a birational morphism is an isomorphism over a dense open subset of the target \stacks{01RN}.
\end{notation*}
	
\begin{defin} \label{abs--defin}
An \df{abstract blow-up square} is a cartesian diagram of schemes
	\[ \tag{abs}
	\xymatrix{
	E \ar[d] \ar[r] & \tilde X \ar[d] \\
	Z  \ar[r] & X
	}
	\]
where $Z\to X$ is a closed immersion, $\tilde X\to X$ is proper, and the induced morphism $\tilde X \setminus E \to X \setminus Z$ is an isomorphism. 
\end{defin}

\begin{defin} \label{rh-cdh-topology--defin}
Let $S$ be a noetherian scheme. The \df{rh-topology} on $\Sch_S$ is the topology generated by Zariski squares and covers $\{Z\to X, \tilde{X}\to X\}$ for every abstract blow-up square (abs) as well as the empty cover of the empty scheme.
The \df{cdh-topology} (completely decomposed h-topology) is the topology generated by Nisnevich squares and abstract blow-up squares as well as the empty cover of the empty scheme.
\end{defin}

\begin{remark} \label{rh-cdh--remark}
The cdh-topology relates to the Nisnevich topology in the same way as the rh-topology relates to the Zariksi topology. Thus a lot of results in the literature concerning the cdh-topology are also valid for the rh-topology. Possible occurences of the Nisnevich topology may be substituted by the Zariski topology. Hence the same proofs apply almost verbatim by exchanging only the terms ``cdh'' by ``rh'', ``Nisnevich'' by ``Zariksi'', and ``\etale{} morphism'' by ``open immersion''. 
\end{remark}

\begin{remark}
Our definition of an abstract blow-up square coincides with the one given in \cite{kst}. Other authors demand instead the weaker condition that the induced morphism $(\tilde X\setminus E)\red \to (X\setminus Z)\red$ on the associated reduced schemes is an isomorphism, e.g. \cite[Def.~12.21]{mvw01}. Indeed, both notions turn out to yield the same topology. To see this, first note that the map $X\red\to X$ is a rh-cover since
	\[
	\xymatrix{
	\emptyset \ar[d] \ar[r] & \emptyset \ar[d] \\
	X\red  \ar[r] & X
	}
	\]
is an abstract blow-up square in the sense of Definition~\ref{abs--defin} since $X\red\to X$ is a closed immersion, $\emptyset\to X$ is proper, and the induced map on the complements $\emptyset\to\emptyset$ is an isomorphism. Now we consider the following situation as indicated in the diagram
	\[
	\xymatrix@R=2.5ex@C=1.2ex{ 
	&&& \tilde X\red \ar [dl] \ar'[d][dd] && \ar[ll] (\tilde X_U)\red \ar[dl] \ar[dd]^\cong \\
	&& \tilde X \ar[dd] && \ar[ll] \tilde X_U \ar[dd] \\
	& Z\red \ar[dl] \ar'[r][rr] && X\red \ar[dl] && \ar'[l][ll]  \ar[dl] U\red \\
	Z \ar[rr] && X && \ar[ll] U  
	}
	\]
where $U := X\setminus Z$ and $\tilde X_U := \tilde X \times_X U$. The morphism $Z \sqcup \tilde X\to X$ is a cover in the sense of \cite{mvw01} but not a priori in the sense of Definition~\ref{rh-cdh-topology--defin}. However, it can be refined by the composition $Z\red \sqcup \tilde X\red \to X\red\to X$ in which both maps are rh-covers.
\end{remark}

\begin{defin} \label{nis-lifting-prop--def}
A morphism of schemes $\tilde{X} \to X$ satisfies the \df{Nisnevich lifting property} iff every point $x\in X$ has a preimage $\tilde{x}\in \tilde{X}$ such that the induced morphism $\kappa(x)\to\kappa(\tilde{x})$ on residue fields is an isomorphism.
\end{defin}

\begin{lemma} \label{generic-iso--lem}
Let $p\colon \tilde{X}\to X$ be a proper map satisfying the Nisnevich lifting property and assume $X$ to be reduced. Then there exists a closed subscheme $X'$ of $\tilde{X}$ such that the restricted map $p|_{X'} \colon X'\to X$ is birational.
\end{lemma}
\begin{proof}
Let $\eta$ be a generic point of $X$. By assumption, there exists a point $\tilde{\eta}$ of $\tilde{X}$ mapping to $\eta$. Since $p$ is a closed map, we have $p\bigl(\,\overline{\{\tilde{\eta}\}}\,\bigr) \supset \overline{p(\{\tilde{\eta}\})} = \overline{\{\eta\}}$ and hence equality holds. Thus the restriction $\overline{\{\tilde{\eta}\}}\to \overline{\{\eta\}}$ is a  morphism between reduced and irreducible schemes inducing an isomorphism on the stalks of the generic points, hence it is birational. Thus setting $X'$ to be the (finite) union of all $\overline{\{\tilde{\eta}\}}$ for all generic points $\eta$ of $X$ does the job.
\end{proof}

\begin{lemma}[{\cite[2.18]{voevodsky-unstable}}]
A proper map is an rh-cover if and only if it satisfies the Nisnevich lifting property.
\end{lemma}

\begin{defin}
A \df{proper rh-cover} is a proper map which is also an rh-cover, \ie{} a proper map satisfying the Nisnevich lifting property. By Lemma~\ref{generic-iso--lem}, every proper rh-cover of a reduced scheme has a refinement by a proper birational rh-cover.
\end{defin}

\begin{cor} \label{rh-covers-agree--cor}
The rh-topology equals the topology which is generated by Zariski covers and by proper rh-covers.
\end{cor}

\begin{defin} \label{rh-excision-discrete--def}
We say that a set-valued presheaf $F$ satisifes \df{rh-excision} iff for every abstract blow-up square (abs) as in Definition~\ref{abs--defin} the induces square
	\carre{F(X)}{F(Z)}{F(\tilde{X})}{F(E)}
is a pullback square.
\end{defin}

\begin{prop}
A Zariski sheaf is an rh-sheaf if and only if it satisfies rh-excision.
\end{prop}
\begin{proof}
The proof is analogous to the proof of the corresponding statement for the Nisnevich topology \cite[12.7]{mvw01}, \cf{} Remark~\ref{rh-cdh--remark}.
\end{proof}

\begin{defin} \label{hrh-hcdh-topology--defin}
The \df{hrh-topology} (honest rh-topology) (\resp{} the \df{hcdh-topology}) is the topology generated by honest blow-up squares
	\[
	\xymatrix{
	E \ar[d] \ar[r] & \Bl_Z(X) \ar[d] \\
	Z  \ar[r] & X
	}
	\]
and Zariski squares (\resp{} Nisnevich squares) as well as the empty cover of the empty scheme.
\end{defin}

\begin{lemma} \label{hcdh=cdh--lemma}
Let $S$ be a noetherian scheme. On $\Sch_S$ the hrh-topology equals the rh-topology and the hcdh-topology equals the cdh-topology.
\end{lemma}
\begin{proof}
Every hrh-cover is an rh-cover. We have to show conversely that every rh-cover can be refined by an hrh-cover. Let $X\in\Sch_S$. It suffices to show that a cover coming from an abstract blow-up square over $X$ can be refined by an hrh-cover. As $\Bl_{X\red}(X)=\emptyset$, the map $X\red\to X$ is an hrh-cover. Hence we can assume that $X$ is reduced since pullbacks of abstract blow-up squares are abstract blow-up squares again. Let $X = X_1 \cup\ldots\cup X_n$ be the decomposition into irreducible components. For a closed subscheme $Z$ of $X$ one has
	\[
	\Bl_Z(X) = \Bl_Z(X_1) \cup\ldots\cup \Bl_Z(X_n).
	\]
If $Z=X_n$ , then $\Bl_{X_n}(X_n)=\emptyset$ and $\Bl_{X_n}(X_i)$ is irreducible for $i\in\{1,\ldots,n-1\}$ \cite[Cor.~13.97]{gw10}. By iteratively blowing up along the irreducible components, we can hence reduce to the case where $X$ is irreducible. Let
	\carre{E}{\tilde X}{Z}{X}
be an abstract blow-up square. As $X$ is irreducible, the complement $U:=X\setminus Z$ is schematically dense in $X$. As $p\colon \tilde X\to X$ is proper and birational, also $\tilde X$ is irreducible and $p\inv(U)$ is schematically dense in $\tilde X$. Thus $p$ is a $U$-modification and a result of Temkin tells us that there exists a $U$-admissible blow-up factoring over $\tilde X$ \cite[Lem.~2.1.5]{temkin08}.\footnote{Even though most parts of Temkin's article \cite{temkin08} deal with characteristic zero, this is not the case for the mentioned result.}
That means that there exists a closed subscheme $Z'$ of $X$ which lies in $X\setminus U=Z$ such that $\Bl_{Z'}(X)\to X$ factors over $\tilde X\to X$. Thus the rh-cover $\{Z\to X,\tilde X\to X\}$ can be refined by the hrh-cover $\{Z'\to X,\Bl_{Z'}(X)\to X\}$ which was to be shown. The second part has the same proof.
\end{proof}

\begin{prop} \label{rh-cover_factorisation--prop}
Let $S$ be a noetherian scheme and let $X\in\Sch_S$. Every rh-cover of $X$ admits a refinement of the form $U \to[f] \tilde X \to[p] X$ where $f$ is a Zariski cover and $p$ is a proper rh-cover.
\end{prop}
\begin{proof}
The proof is analogous to the corresponding result for the cdh-topology due to Suslin-Voevodsky \cite[5.9]{sv00} or the proof given in \cite[12.27,12.28]{mvw01}, \cf{} Remark~\ref{rh-cdh--remark}. 
\end{proof}

The following theorem and its proof are just rh-variants of the corresponding statement for the cdh-topology by Kerz-Strunk-Tamme \cite[6.3]{kst}. The theorem goes back to Haesemeyer \cite{haesemeyer04}. Another recent proof for the cdh-topology which also works for the rh-topology was recently given by Kelly-Morrow \cite[3.4]{kelly-morrow}. 

\begin{notation*}
Let $S$ be a scheme and let $\mathrm{Sh}_\Ab(\Sch_S^\rh)$ be the category of rh-sheaves  on $\Sch_S$ with values in abelian groups. Its inclusion into the category $\mathrm{PSh}_\Ab(\Sch_S)$ of presheaves on $\Sch_S$ with values in abelian groups admits an exact left adjoint $\arh$. Similarly, the inclusion $\Sh_\Sp(\Sch_S^\rh) \inj \PSh_\Sp(\Sch_S)$ of rh-sheaves on $\Sch_S$ with values in the $\infty$-category of spectra $\Sp$ admits an exact left adjoint $\Lrh$.
\end{notation*}

\begin{thm} \label{Lrh_K=KH--thm}
Let $S$ be a finite-dimensional noetherian scheme. Then the canonical maps of rh-sheaves with values in spectra on $\Sch_S$
	\[
	\Lrh\K_{\geq 0} \To \Lrh\K \To \KH
	\]
are equivalences.
\end{thm}

In the proof of the theorem we will make use of the following lemma.

\begin{lemma} \label{Lrh=0--lemma}
Let $S$ be a finite dimensional noetherian scheme and let $F$ be a presheaf of abelian groups on $\Sch_S$. Assume that
\begin{enumerate}
\item for every reduced affine scheme $X$ and every element $\alpha\in F(X)$ there exists a proper birational morphism $X'\to X$ such that $F(X)\to F(X')$ maps $\alpha$ to zero and that
\item $F(Z)=0$ if $\dim(Z)=0$ and $Z$ is reduced.
\end{enumerate} 
Then $\arh F =0$.
\end{lemma}
\begin{proof}
It suffices to show that for every affine scheme $X$ the map
	\[ 
	F(X) \To \arh F(X)
	\]
vanishes as this implies that the rh-stalks are zero.
Consider the diagram
	\carre{F(X)}{\arh F(X)}{F(X\red)}{\arh F(X\red).}
As $X\red \to X$ is an rh-cover, the right vertical map is an isomorphism. Thus we can assume that $X$ is reduced. For any $\alpha\in F(X)$ there exists, by condition (i), a proper birational morphism $f\colon X'\to X$ such that $\alpha$ maps to zero in $F(X')$. Let $U$ be an open dense subscheme of $X$ over which $f$ is an isomorphism and set $Z := (X\setminus U)\red$. Then $\dim(Z) < \dim(X)$ as $U$ is dense. By condition (ii) we can argue by induction that $\alpha$ maps to zero in $\arh F(Z)$. Since $X'\sqcup Z \to X$ is an rh-cover by construction, $\alpha$ vanishes on some rh-cover of $Y$. Thus the map $F(X) \to \arh F(X)$ maps every element to zero which finishes the proof.
\end{proof}

\begin{example} \label{Lrh_K_-i--example}
\textbf{(i)}
For $i<0$, the functor $F=\K_i$ satisfies the conditions of Lem\-ma~\ref{Lrh=0--lemma}. Zero-dimensional reduced schemes are regular and hence their negative K-theory vanishes, and condition (i) was proven by Kerz-Strunk \cite[Prop.~5]{ks16}.

\textbf{(ii)} Another example is the functor $F = N\K_i$ for $i\in\Z$ which is defined by $N\K_i(X) := \K_i(\mathbf{A}^1_X) / \K_i(X)$. The K-theory of regular schemes is homotopy invariant, and condition (i) was proven by Kerz-Strunk-Tamme \cite[Prop.~6.4]{kst}.
\end{example}

The following proposition is just a recollection from the literature which will be used in the proof of Theorem~\ref{Lrh_K=KH--thm}.

\begin{prop}[Voevodsky, Asok-Hoyois-Wendt] \label{rh-hypercomplete--prop}
Let $S$ be a noetherian scheme of finite dimension. Then:
\begin{enumerate}
 \item The \inftop{} $\Sh^\rh(\Sch_S)$ of space-valued rh-sheaves on $\Sch_S$ is hypercomplete.
 \item The \infcat{} $\Sh^\rh_\Sp(\Sch_S)$ of spectrum-valued rh-sheaves on $\Sch_S$ is left-complete.
 \item A map in $\Sh^\rh_\Sp(\Sch_S)$ is an equivalence if and only if it is an equivalence on stalks.
\end{enumerate} 
\end{prop}
\begin{proof}
 The rh-topology is induced by a cd-structure \cite[Def.~2.1]{voevodsky-cd} which is complete, regular, and bounded \cite[Thm.~2.2]{voevodsky-unstable}. Hence a space-valued presheaf is a hypercomplete rh-sheaf if and only if it is rh-excisive; this follows from \cite[Lem.~3.5]{voevodsky-cd}. On the other hand, as the cd-structure is complete and regular, a space-valued presheaf is an rh-sheaf if and only if it is rh-excisive; this follows from \cite[Thm.~3.2.5]{ahw1}. Together this implies (i), \cf{} \cite[Rem.~3.2.6]{ahw1}.
The \infcat{} $\Sh^\rh_\Sp(\Sch_S)$ is equivalent to the \infcat{} $\Sh^\rh_\Sp(\Sh^\rh(\Sch_S))$ of sheaves of spectra on the \inftop{} $\Sh^\rh(\Sch_S)$ \cite[1.3.1.7]{sag}. Hence the \inftop{} $\Sh^\rh_\Sp(\Sch_S)$ of connective objects is Postnikov complete. As we can write every object $F$ as the colimit $\colim_{n\in\N} F_{\geq -n}$ of objects which are (up to a shift) connective, this implies (ii) and (iii).
\end{proof}

\begin{proof}[Proof of Theorem~\ref{Lrh_K=KH--thm}]
As the \inftop{} of space-valued sheaves on $\Sch_S^\rh$ is hypercomplete, we can test the desired equivalences on stalks (Proposition~\ref{rh-hypercomplete--prop}). Since spheres are compact, taking homotopy groups commutes with filtered colimits and we can check on the sheaves of homotopy groups of the stalks whether the maps are equivalences. Thus the first equivalence follows directly by applying Lemma~\ref{Lrh=0--lemma} and Example~\ref{Lrh_K_-i--example}~(i) and since the connective cover has isomorphic non-negative homotopy groups.

For the second equivalence we assume for a moment the existence of a weakly convergent spectral sequence 
	\[
	E^1_{p,q} = \arh N^p\K_q \Rightarrow \arh \KH_{p+q}
	\]
in $\Shh_\Ab(\Sch_S^\rh)$. It suffices to show that $\arh N^p\K_q =0$ for $p\geq 1$ which follows from Lemma~\ref{Lrh=0--lemma} and Example~\ref{Lrh_K_-i--example}~(ii). Thus the proof is finished by the following lemma.
\end{proof}

\begin{lemma}
There is a weakly convergent spectral sequence
	\[
	E^1_{p,q} = \arh N^p\K_q \Rightarrow \arh \KH_{p+q}
	\]
of rh-sheaves of abelian groups on $\Sch_X$.
\end{lemma}
\begin{proof}
For every ring $R$ there is a weakly convergent spectral sequence \cite[IV.12.3]{weibel}
	\[
	E^1_{p,q} = N^p\K_q(R) \Rightarrow \KH_{p+q}(R).
	\]
This yields a spectral sequence $E^1_{p,q}=N^p\K_q$ on the associated presheaves of abelian groups on $\Sch_X$ and hence a spectral sequence $E^1_{p,q} = \arh N^p\K_q$ of the associated rh-sheafifications. We have to check that the latter one converges to $\arh \KH_{p+q}$. This can be tested on rh-stalks which are filtered colimits of the weakly convergent spectral sequence above. As filtered colimits commute with colimits and finite limits, a filtered colimit of weakly convergent spectral sequence yields a weakly convergent spectral sequence. Hence we are done.
\end{proof}

\begin{thm} \label{K_-d=Hrh^d--thm}
Let $X$ be a $d$-dimensional noetherian scheme. Then there exists a canonical isomorphism
	\[
	\K_{-d}(X) \cong \Hrh^d(X;\Z).
	\]
\end{thm}
\begin{proof}
As $\KH$ is an rh-sheaf, the Zariski descent spectral sequence appears as
	\[
	E_2^{p,q} = \Hrh^p(X,\arh (\K_{-q})) \Rightarrow \KH_{-p-q}(X).
	\]
We know the following:
\begin{itemize}
	\item $E_2^{p,q}=0$ for $p>d$ as the rh-cohomological dimension is bounded by the dimension \cite[2.27]{voevodsky-unstable}.
	\item $\arh (\K_{\geq 0})_0 = \Z$ since $(\K_{\geq 0})_0(R) = \K_0(R)=\Z$ for any local ring $R$.
	\item $\K_{-d}(X) \cong \KH_{-d}(X)$ by $\K_{-d}$-regularity and the vanishing of $\K_{-i}$ for $i>\dim(X)$ \cite[Thm.~B]{kst} together with the spectral sequence relating K-theory and KH-theory \cite[IV.12.3]{weibel}.
	\item $\arh \K_{-q} = 0$ for $q>0$ by Lemma~\ref{Lrh=0--lemma} and Example~\ref{Lrh_K_-i--example}.
\end{itemize}
This implies that in the $2$-page only the term $E_2^{d,0} = \Hrh^d(X;\Z)$ contributes on the line $-p-q=-d$ and that already $E_2^{d,0} = E_\infty^{d,0}$ since all differentials of all $E_i^{d,0}$ for $i\geq 2$ come from or go to zero. Hence the theorem follows.
\end{proof}

\bibliography{literature}

\providecommand{\bysame}{\leavevmode\hbox to3em{\hrulefill}\thinspace}
\providecommand{\MR}{\relax\ifhmode\unskip\space\fi MR }
\providecommand{\MRhref}[2]{%
  \href{http://www.ams.org/mathscinet-getitem?mr=#1}{#2}
}
\providecommand{\href}[2]{#2}
\begin{thebibliography}{CHSW08}

\bibitem[Abb10]{egr}
Ahmed Abbes, \emph{{\'E}l\'ements de g\'eom\'etrie rigide. {V}olume {I}},
  Progress in Mathematics, vol. 286, Birkh\"auser/Springer Basel AG, Basel,
  2010, Construction et \'etude g\'eom\'etrique des espaces rigides.

\bibitem[AHW17]{ahw1}
Aravind Asok, Marc Hoyois, and Matthias Wendt, \emph{Affine representability
  results in {$\mathbf{A}^1$}-homotopy theory {I}: vector bundles}, Duke Math.
  J. \textbf{166} (2017), no.~10, 1923--1953.

\bibitem[Bei14]{beilinson}
Alexander Beilinson, \emph{Relative continuous {$K$}-theory and cyclic
  homology}, M\"{u}nster J. Math. \textbf{7} (2014), no.~1, 51--81.

\bibitem[BEK14a]{bek-zero}
Spencer Bloch, H\'{e}l\`ene Esnault, and Moritz Kerz, \emph{Deformation of
  algebraic cycle classes in characteristic zero}, Algebr. Geom. \textbf{1}
  (2014), no.~3, 290--310. \MR{3238152}

\bibitem[BEK14b]{bek-p-adic}
\bysame, \emph{{$p$}-adic deformation of algebraic cycle classes}, Invent.
  Math. \textbf{195} (2014), no.~3, 673--722.

\bibitem[Ber90]{berkovich90}
Vladimir~G. Berkovich, \emph{Spectral theory and analytic geometry over
  non-{A}rchimedean fields}, Mathematical Surveys and Monographs, vol.~33,
  American Mathematical Society, Providence, RI, 1990.

\bibitem[Ber99]{berkovich99}
\bysame, \emph{Smooth {$p$}-adic analytic spaces are locally contractible},
  Invent. Math. \textbf{137} (1999), no.~1, 1--84.

\bibitem[BGR84]{bgr}
Siegfried Bosch, Ulrich G\"{u}ntzer, and Reinhold Remmert,
  \emph{Non-{A}rchimedean {A}nalysis}, Grundlehren der Mathematischen
  Wissenschaften, vol. 261, Springer-Verlag, Berlin, 1984.

\bibitem[BGT13]{bgt13}
Andrew~J. Blumberg, David Gepner, and Gon\c{c}alo Tabuada, \emph{A universal
  characterization of higher algebraic {$K$}-theory}, Geom. Topol. \textbf{17}
  (2013), no.~2, 733--838.

\bibitem[Bos14]{bosch}
Siegfried Bosch, \emph{Lectures on formal and rigid geometry}, Lecture Notes in
  Mathematics, vol. 2105, Springer, Cham, 2014.

\bibitem[Bre97]{bredon-sheaf}
Glen~E. Bredon, \emph{Sheaf theory}, second ed., Graduate Texts in Mathematics,
  vol. 170, Springer-Verlag, New York, 1997.

\bibitem[Cal85]{calvo}
Adina Calvo, \emph{{$K$}-th\'{e}orie des anneaux ultram\'{e}triques}, C. R.
  Acad. Sci. Paris S\'{e}r. I Math. \textbf{300} (1985), no.~14, 459--462.

\bibitem[CHSW08]{chsw08}
Guillermo Corti{\~{n}}as, Christian Haesemeyer, Marco Schlichting, and
  Charles~A. Weibel, \emph{Cyclic homology, cdh-cohomology and negative
  {$K$}-theory}, Ann. of Math. (2) \textbf{167} (2008), no.~2, 549--573.

\bibitem[Dun98]{dundas-cont}
Bj\o rn~Ian Dundas, \emph{Continuity of {$K$}-theory: an example in equal
  characteristics}, Proc. Amer. Math. Soc. \textbf{126} (1998), no.~5,
  1287--1291.

\bibitem[FK18]{fuji-kato}
Kazuhiro Fujiwara and Fumiharo Kato, \emph{Foundations of {R}igid {G}eometry
  {I}}, Monographs in Mathematics, vol.~7, EMS, 2018.

\bibitem[FvdP04]{fresnel-vdp}
Jean Fresnel and Marius van~der Put, \emph{Rigid analytic geometry and its
  applications}, Progress in Mathematics, vol. 218, Birkh\"{a}user Boston,
  Inc., Boston, MA, 2004. \MR{2014891}

\bibitem[GH06a]{geisser-hesselholt-K-TC}
Thomas Geisser and Lars Hesselholt, \emph{On the {$K$}-theory and topological
  cyclic homology of smooth schemes over a discrete valuation ring}, Trans.
  Amer. Math. Soc. \textbf{358} (2006), no.~1, 131--145.

\bibitem[GH06b]{geisser-hesselholt-K-F_p}
\bysame, \emph{On the {$K$}-theory of complete regular local
  {$\mathbf{F}_p$}-algebras}, Topology \textbf{45} (2006), no.~3, 475--493.

\bibitem[GK15]{gabber-kelly}
Ofer Gabber and Shane Kelly, \emph{Points in algebraic geometry}, J. Pure Appl.
  Algebra \textbf{219} (2015), no.~10, 4667--4680.

\bibitem[GL01]{goodwillie-lichtenbaum}
Thomas~G. Goodwillie and Stephen Lichtenbaum, \emph{A cohomological bound for
  the {$h$}-topology}, Amer. J. Math. \textbf{123} (2001), no.~3, 425--443.

\bibitem[GW10]{gw10}
Ulrich G\"ortz and Torsten Wedhorn, \emph{Algebraic geometry {I}}, Advanced
  Lectures in Mathematics, Vieweg + Teubner, Wiesbaden, 2010, Schemes with
  examples and exercises.

\bibitem[Hae04]{haesemeyer04}
Christian Haesemeyer, \emph{Descent properties of homotopy {$K$}-theory}, Duke
  Math. J. \textbf{125} (2004), no.~3, 589--620.

\bibitem[Hak72]{hakim}
Monique Hakim, \emph{Topos annel\'{e}s et sch\'{e}mas relatifs},
  Springer-Verlag, Berlin-New York, 1972, Ergebnisse der Mathematik und ihrer
  Grenzgebiete, Band 64.

\bibitem[HL16]{hrushovski-loeser}
Ehud Hrushovski and Fran\c{c}ois Loeser, \emph{Non-archimedean tame topology
  and stably dominated types}, Annals of Mathematics Studies, vol. 192,
  Princeton University Press, Princeton, NJ, 2016.

\bibitem[Hoy14]{hoyois-quadratic}
Marc Hoyois, \emph{A quadratic refinement of the
  {G}rothendieck-{L}efschetz-{V}erdier trace formula}, Algebr. Geom. Topol.
  \textbf{14} (2014), no.~6, 3603--3658.

\bibitem[Hub93]{huber93}
Roland Huber, \emph{Continuous valuations}, Math. Z. \textbf{212} (1993),
  no.~3, 455--477. \MR{1207303}

\bibitem[Hub94]{huber94}
\bysame, \emph{A generalization of formal schemes and rigid analytic
  varieties}, Math. Z. \textbf{217} (1994), no.~4, 513--551.

\bibitem[Ker18]{kerz-icm}
Moritz Kerz, \emph{On negative algebraic {$K$}-groups}, Proc.~Int.~Cong.~of
  Math. \textbf{1} (2018), 163--172.

\bibitem[KM18]{kelly-morrow}
Shane Kelly and Matthew Morrow, \emph{K-theory of valuation rings},
  arXiv:1810.12203, 2018.

\bibitem[KS17]{ks16}
Moritz Kerz and Florian Strunk, \emph{On the vanishing of negative homotopy
  {$K$}-theory}, J. Pure Appl. Algebra \textbf{221} (2017), no.~7, 1641--1644.

\bibitem[KST18]{kst}
Moritz Kerz, Florian Strunk, and Georg Tamme, \emph{Algebraic {K}-theory and
  descent for blow-ups}, Invent. Math. \textbf{211} (2018), no.~2, 523--577.

\bibitem[KST19]{kst18}
Moritz Kerz, Shuji Saito, and Georg Tamme, \emph{K-theory of non-archimedean
  rings. {I}}, Doc. Math. \textbf{24} (2019), 1365–1411.

\bibitem[KV71]{kv71}
Max Karoubi and Orlando Villamayor, \emph{{$K$}-th\'{e}orie alg\'{e}brique et
  {$K$}-th\'{e}orie topologique. {I}}, Math. Scand. \textbf{28} (1971),
  265--307 (1972).

\bibitem[Lur09]{htt}
Jacob Lurie, \emph{Higher topos theory}, Annals of Mathematics Studies, vol.
  170, Princeton University Press, Princeton, NJ, 2009.

\bibitem[Lur17]{ha}
\bysame, \emph{Higher {A}lgebra}, Preprint (version dated September 18, 2017),
  available at
  \href{http://www.math.harvard.edu/~lurie/}{www.math.harvard.edu/\~{}lurie},
  2017.

\bibitem[Lur18]{sag}
\bysame, \emph{Spectral {A}lgebraic {G}eometry ({U}nder {C}onstruction!)},
  Preprint (version dated February 3, 2018), available at
  \href{http://www.math.harvard.edu/~lurie/}{www.math.harvard.edu/\~{}lurie},
  2018.

\bibitem[L{\"{u}}t16]{luetkebohmert-curves}
Werner L{\"{u}}tkebohmert, \emph{Rigid geometry of curves and their
  {J}acobians}, Ergebnisse der Mathematik und ihrer Grenzgebiete. 3. Folge. A
  Series of Modern Surveys in Mathematics [Results in Mathematics and Related
  Areas. 3rd Series. A Series of Modern Surveys in Mathematics], vol.~61,
  Springer, Cham, 2016.

\bibitem[Mor16]{morrow}
Matthew Morrow, \emph{A historical overview of pro cdh descent in algebraic
  {$K$}-theory and its relation to rigid analytic varieties}, arXiv:1612.00418,
  2016.

\bibitem[MVW06]{mvw01}
Carlo Mazza, Vladimir Voevodsky, and Charles~A. Weibel, \emph{Lecture notes on
  motivic cohomology}, Clay Mathematics Monographs, vol.~2, American
  Mathematical Society, Providence, RI; Clay Mathematics Institute, Cambridge,
  MA, 2006.

\bibitem[{Nic}14]{nicaise}
Johannes {Nicaise}, \emph{Berkovich skeleta and birational geometry},
  arXiv:1409.5229, 2014.

\bibitem[Sch12]{scholze-perfectoid}
Peter Scholze, \emph{Perfectoid spaces}, Publ. Math. Inst. Hautes \'{E}tudes
  Sci. \textbf{116} (2012), 245--313.

\bibitem[{Sta}19]{stacks-project}
The {Stacks Project Authors}, \emph{\textit{Stacks Project}},
  \url{http://stacks.math.columbia.edu}, 2019.

\bibitem[SV00]{sv00}
Andrei Suslin and Vladimir Voevodsky, \emph{Bloch-{K}ato conjecture and motivic
  cohomology with finite coefficients}, The arithmetic and geometry of
  algebraic cycles ({B}anff, {AB}, 1998), NATO Sci. Ser. C Math. Phys. Sci.,
  vol. 548, Kluwer Acad. Publ., Dordrecht, 2000, pp.~117--189.

\bibitem[Tat71]{tate71}
John Tate, \emph{Rigid analytic spaces}, Invent. Math. \textbf{12} (1971),
  257--289.

\bibitem[Tem08]{temkin08}
Michael Temkin, \emph{Desingularization of quasi-excellent schemes in
  characteristic zero}, Adv. Math. \textbf{219} (2008), no.~2, 488--522.

\bibitem[TT90]{tt90}
Robert~W. Thomason and Thomas Trobaugh, \emph{Higher algebraic {$K$}-theory of
  schemes and of derived categories}, The {G}rothendieck {F}estschrift, {V}ol.\
  {III}, Progr. Math., vol.~88, Birkh\"auser Boston, Boston, MA, 1990,
  pp.~247--435.

\bibitem[vdPS95]{put-schneider-points}
Marius van~der Put and Peter Schneider, \emph{Points and topologies in rigid
  geometry}, Math. Ann. \textbf{302} (1995), no.~1, 81--103.

\bibitem[Voe10a]{voevodsky-cd}
Vladimir Voevodsky, \emph{Homotopy theory of simplicial sheaves in completely
  decomposable topologies}, J. Pure Appl. Algebra \textbf{214} (2010), no.~8,
  1384--1398.

\bibitem[Voe10b]{voevodsky-unstable}
\bysame, \emph{Unstable motivic homotopy categories in {N}isnevich and
  cdh-topologies}, J. Pure Appl. Algebra \textbf{214} (2010), no.~8,
  1399--1406.

\bibitem[Wag76a]{wagoner-cont}
John~B. Wagoner, \emph{Continuous cohomology and {$p$}-adic {$K$}-theory},
  Algebraic {K}-theory ({P}roc. {C}onf., {N}orthwestern {U}niv., {E}vanston,
  {I}ll., 1976), Lecture Notes in Math., vol. 551, Springer, Berlin, 1976,
  pp.~241--248,.

\bibitem[Wag76b]{wagoner-cont-dvr}
\bysame, \emph{Delooping the continuous {$K$}-theory of a valuation ring},
  Pacific J. Math. \textbf{65} (1976), no.~2, 533--538.

\bibitem[Wei80]{weibel-analytic}
Charles~A. Weibel, \emph{{$K$}-theory and analytic isomorphisms}, Invent. Math.
  \textbf{61} (1980), no.~2, 177--197.

\bibitem[Wei01]{weibel-normal}
\bysame, \emph{The negative {$K$}-theory of normal surfaces}, Duke Math. J.
  \textbf{108} (2001), no.~1, 1--35.

\bibitem[Wei13]{weibel}
\bysame, \emph{The {$K$}-book}, Graduate Studies in Mathematics, vol. 145,
  American Mathematical Society, Providence, RI, 2013, An Introduction to
  Algebraic $K$-theory.

\end{thebibliography}
\bibliographystyle{amsalpha}
\end{document}